\documentclass{amsart}
\usepackage[utf8]{inputenc}
\usepackage{amssymb,latexsym}
\usepackage{amsmath}
\usepackage{graphicx}

\title[Fractionally balanced hypergraphs]
{fractionally balanced hypergraphs and rainbow KKM theorems}

\author[Aharoni]{Ron Aharoni} \thanks{R. Aharoni:
Department of Mathematics, Technion, Israel and MIPT, Dolgoprudny, Russia.
\url{ra@tx.technion.ac.il}.
R. Aharoni was Supported by the Israel Science Foundation (ISF) grant no.\ 2023464 and the Discount Bank Chair at the Technion. This paper is part of a project that has received funding from the European Union's Horizon 2020 research and innovation programme under the Marie Sklodowska-Curie grant agreement no.\ 823748.
This work was supported by the Russian Federation Government in the framework of MegaGrant no. 075-15-2019-1926 
 when Ron Aharoni worked on Sections 6 through 8 of the paper.
}

\author[Berger]{Eli Berger} \thanks{E. Berger: Department of Mathematics, University of Haifa,  Israel. \url{berger@math.haifa.ac.il}.}

\author[Briggs]{Joseph Briggs} \thanks{J. Briggs: Department of Mathematics and Statistics, Auburn University, USA. \url{jgb0059@auburn.edu}}

\author[Segal-Halevi]{Erel Segal-Halevi} \thanks{E. Segal-Halevi:
Department of Computer Science, Ariel University, Israel. 
\url{erelsgl@gmail.com}.
Erel Segal-Halevi was supported by the Israel Science Foundation (grant no. 712/20).}

\author[Zerbib]{Shira Zerbib}\thanks{S. Zerbib: Department of Mathematics, Iowa State University, USA.  \url{zerbib@iastate.edu}. S. Zerbib was supported by NSF grant DMS-1953929. \\ R. Aharoni, E. Berger and S. Zerbib are supported by BSF grant no. 2016077.}

\usepackage{amsthm,amssymb,amsmath,enumerate,graphicx, tikz}
\usetikzlibrary{arrows.meta,positioning,chains,fit,shapes,calc}
\usepackage{amscd}
\usepackage{setspace}
\usepackage{comment}
\usepackage{hyperref}

\usepackage[square,numbers]{natbib}

\newtheorem{theorem}{Theorem}[section]

\newtheorem*{theorem*}{Theorem}
\newtheorem{proposition}[theorem]{Proposition}

\newtheorem{lemma}[theorem]{Lemma}
\newtheorem{observation}[theorem]{Observation}
\newtheorem{corollary}[theorem]{Corollary}
\newtheorem{claim}[theorem]{Claim}
\newtheorem{conjecture}[theorem]{Conjecture}
\theoremstyle{definition}
\newtheorem{definition}[theorem]{Definition}
\newtheorem*{definition*}{Definition}
\newtheorem{notation}[theorem]{Notation}
\theoremstyle{remark}
\newtheorem{remark}[theorem]{Remark}
\newtheorem{question}[theorem]{Question}
\newtheorem{example}[theorem]{Example}

\newtheorem{convention}[theorem]{Convention}

\newcommand{\cf}{\mathcal{F}}
\newcommand{\cp}{\mathcal{P}}

\newcommand{\F}{\mathcal{F}}

\newcommand{\ca}{\mathcal{A}}
\newcommand{\cj}{\mathcal{J}}

\newcommand{\cv}{\mathcal{V}}

\newcommand{\vj}{\vec{j}}

\newcommand{\U}{\mathcal{U}}
\newcommand{\R}{\mathbb{R}}
\newcommand{\Z}{\mathbb{Z}}
\newcommand{\bh}{\mathcal{BH}}

\newcommand{\C}{\mathcal{C}}

\newcommand{\cc}{\mathcal{C}}

\newcommand{\M}{\mathcal{M}}

\newcommand{\ci}{\mathcal{I}}
\newcommand{\cm}{\mathcal{M}}

\newcommand{\dnm}{\Delta_{n-1}}

\newcommand{\ad}{\textsc{ad}}
\newcommand{\adu}{\textsc{adu}}
\newcommand{\bm}{\textsc{bm}}

\newcommand{\im}{\textsc{im}}

\newcommand{\floor}[1]{\left\lfloor #1 \right\rfloor}
\newcommand{\ceil}[1]{\left\lceil #1 \right\rceil}

\newcommand{\conv}{\operatorname{conv}}

\newcommand{\supp}{\operatorname{supp}}

\definecolor{ForestGreen}{rgb}{.13,.54,.13}
\newcommand{\er}[1]{\textcolor{ForestGreen}{#1}}
\newcommand{\erel}[1]{\textcolor{ForestGreen}{(\textbf{Erel says:} #1)}}

\newcommand{\joe}[1]{\textcolor{purple}{(\textbf{Joe says:} #1)}}

\newcommand{\HiddenExplanation}[1]{}

\numberwithin{equation}{section}

\usepackage{wrapfig}

\begin{document}

\begin{abstract}
A $d$-partite hypergraph is called {\em fractionally balanced} if there exists a non-negative,  not identically zero, function on its edge set that has constant  degrees in each vertex side. 
 Using a topological version of Hall's theorem we prove 
 lower bounds on the matching number of such hypergraphs. These bounds yield 
  rainbow versions of the KKM theorem for products of simplices, which in turn are used to obtain some results on 
 multiple-cake division,  and on rainbow matchings in families of $d$-intervals. 
   \end{abstract}

\maketitle
\section{Introduction}
In bipartite graphs, the existence of a perfect fractional matching implies the existence of a perfect matching \cite{konig1916graphok}. Otherwise put, a non-empty regular bipartite graph has a perfect matching. This is not true for $d$-partite hypergraphs, $d>2$. 
\begin{wrapfigure}{r}{2cm}
\begin{center}
\def\height{0.5}
\begin{tikzpicture}[
scale=3,
]
\node[draw,ellipse] (a1) at (0,0) {$a_0$};
\node[draw,ellipse] (a2) at (0.4,0) {$a_1$};
\node[draw,ellipse,below=\height of a1] (b1)  {$b_0$};
\node[draw,ellipse,below=\height of a2] (b2)  {$b_1$};
\node[draw,ellipse,below=\height of b1] (c1) {$c_0$};
\node[draw,ellipse,below=\height of b2] (c2)  {$c_1$};
\draw [ultra thick] (a1) -- (b1) -- (c1);
\draw [ultra thick] (a1) -- (b2) -- (c2);
\draw (a2) -- (b1) -- (c2);
\draw (a2) -- (b2) -- (c1);
\end{tikzpicture}
\end{center}
\end{wrapfigure}
For example, the famous Pasch hypergraph, a $2 \times 2 \times 2$ hypergraph on $A \times B \times C$, with edge set $\{(a_i, b_j, c_k)\mid i+j+k\equiv 0 ~(\bmod~ 2)\}$, is regular and has matching number $1$ (namely, it is intersecting).  F\"uredi  showed that in a sense this is the worst case. The following is a special case of a theorem of F\"urdei \cite{furedi1981maximum}, relating the matching and fractional matching numbers of a $d$-regular hypergraph (these terms are defined below): 

\begin{theorem}\label{furedi}
A regular $d$-partite hypergraph with sides of size $n$ has a matching of size at least $\frac{n}{d-1}$.
In particular, a regular $3$-partite $3$-uniform $n \times n \times n$ hypergraph has a matching of size at least $\frac{n}{2}$.
\end{theorem} 

   Motivated by problems on fair division of multiple cakes (see Section \ref{sec:cakes} below) we wish to prove similar  results on $3$-partite hypergraphs with not necessarily equal size sides. We shall also study the   case of $d$-partite hypergraphs, for general $d$. 
When the sides are not of equal size the hypergraph cannot be regular, but it can be regular on each side separately. If this is true for $H$ after possibly duplicating some of its edges, we call it ``fractionally balanced''. Equivalently, a $d$-partite hypergraph is fractionally-balanced if there exists a system of non-negative weights on its edges, not all zero, with constant degrees on the vertices in every side.%
\footnote{
The equivalence is obvious if the weights are restricted to be rational.
Allowing irrational weights does not change the set of fractionally-balanced hypergraphs.
This is because the set of balanced weight functions is a cone in $\mathbb{R}^E$ defined by hyperplanes with rational coefficients. Therefore, if it contains a non-zero point, it also contains a non-zero rational point.
}

Let us first cement notation. A hypergraph will be identified with its edge set. A {\em matching} in a hypergraph $H$ is a set of disjoint edges. 
The {\em matching number} $\nu(H)$ is the largest size of a matching in $H$. A {\em fractional matching} in $H$ is a non-negative function $f$ on $H$ such that for every vertex $v$, 
$$\deg_f(v) := \sum_{e\in H,e\ni v} f(e) \le 1.$$ The {\em fractional matching number} $\nu^*(H)$ is the maximum of  $|f|:=\sum_{e\in H} f(e)$ over all fractional matchings $f$ of $H$. 

A hypergraph $H$ is {\em $d$-partite} if its vertex set can be partitioned as $V_1\cup \dots \cup V_d$, in such a way that $|e\cap V_t| = 1$  for every $e\in H$ and  $t\in [d]$. 
We are tacitly assuming that the partition is given and fixed, even though there may be more than one partition satisfying the condition. 
The sets $V_t$ are called the {\em  sides} of $H$. 

\begin{definition}[Fractionally balanced hypergraphs]\label{def:balanced}
Given a $d$-partite hypergraph $H$ with sides $V_1,\ldots,V_d$:

(a) A function $f: H \to \mathbb{R}_{\geq 0}$ is said to be {\em  balanced } if 
it has constant degrees  on every side $V_t$, namely $\deg_f(v)=|f|/|V_t|$ for all $v \in V_t$, for all $t\in[d]$.

(b) $H$ is called $(n_1, \ldots,n_d)$-{\em fractionally balanced} if $|V_t|=n_t$ for all $t\in[d]$, and there exists a balanced nonzero function $f: H \to \mathbb{R}_{\geq 0}$.
\end{definition}

\begin{example}

(a) In a bipartite graph with $n$ vertices in each side, every perfect fractional matching is a balanced weight function; therefore, every graph admitting such a matching is $(n,n)$-fractionally-balanced.

(b) 
The Pasch hypergraph shown above is $(2,2,2)$-fractionally balanced, by a weight function assigning a weight of $1$ to every edge. 
Every complete tripartite hypergraph is fractionally-balanced by a similar weight function.

(c) A slightly less trivial example is shown below. It is $(3,4,4)$-fractionally balanced, by assigning a weight of $1/2$ to every thick edge and $1/4$ to every thin edge.

\begin{center}
\def\aline{1}
\def\bline{0.53}
\def\cline{0}
\def\height{0.7}
\begin{tikzpicture}[scale=3]
\node[draw,ellipse] (a1) at (0.0,\aline) {$a_1$};
\node[draw,ellipse] (a2) at (0.4,\aline) {$a_2$};
\node[draw,ellipse] (a3) at (0.8,\aline) {$a_3$};
\node[draw,ellipse,below=\height of a1] (b1)  {$b_1$};
\node[draw,ellipse,below=\height of a2] (b2)  {$b_2$};
\node[draw,ellipse,below=\height of a3] (b3)  {$b_3$};
\node[draw,ellipse] (b4) at (1.2,\bline)  {$b_4$};
\node[draw,ellipse,below=\height of b1] (c1) {$c_1$};
\node[draw,ellipse,below=\height of b2] (c2)  {$c_2$};
\node[draw,ellipse,below=\height of b3] (c3)  {$c_3$};
\node[draw,ellipse,below=\height of b4] (c4)  {$c_4$};
\draw [ultra thick] (a1) -- (b1) -- (c2);
\draw [ultra thick] (a1) -- (b2) -- (c1);
\draw [ultra thick] (a2) -- (b3) -- (c4);
\draw [ultra thick] (a2) -- (b4) -- (c3);

\draw (a3) -- (b1) -- (c1);
\draw (a3) -- (b2) -- (c2);
\draw (a3) -- (b3) -- (c3);
\draw (a3) -- (b4) -- (c4);
\end{tikzpicture}
\end{center}
\end{example}

\HiddenExplanation{
	\begin{observation}\label{obs:nu*}
	If $H$ is an $(n_1, \ldots ,n_d)$-fractionally balanced hypergraph then $\nu^*(H)=\min_{t \in [d]}n_t$.
	\end{observation}
	Let $i := \arg\min_{i} n_i$.
	Scale the function $f$ such that $\deg_f(v)=1$ for all $v\in V_i$. Then $\deg_f(v)\leq 1$ for all $v\in V$, so $f$ is a fractional matching and $|f|=n_i$.
}

\begin{notation}
Let $\bm(n_1, n_2, \ldots ,n_d)$ denote the largest integer $m$ such that 
every $(n_1, n_2, \ldots ,n_d)$-fractionally balanced  hypergraph contains a matching of size $m$.
\end{notation}
The two results quoted above, about bipartite and $3$-partite hypergraphs, say, in this notation, that $\bm(n, n) \geq n$ and 
$\bm(n, n, n) \geq \lceil\frac{n}{2}\rceil$.

We shall prove:
\begin{theorem}
\label{thm:tri-lower}
For every
positive integers $n_1\leq n_2\leq n_3$:
\begin{align*}
(a)&&
\bm
(n_1, n_2, n_3)\geq 
\min\bigg(
n_1,~
\ceil{\frac{\floor{2 n_3/n_2} n_2 }{\floor{2 n_3/n_2}+2}}
\bigg)
\\
(b)&&
\bm
(n_1, n_2, n_3)\geq 
\min\bigg(
n_1,~
\ceil{\frac{2 n_3}{\ceil{2 n_3/n_2}+2}}
\bigg)
\end{align*}
In particular, if $n_3/n_2$ is an integer, then
\begin{align*}
\bm
(n_1, n_2, n_3)
&\geq 
\min\bigg(
n_1,~
\ceil{\frac{1}{1/n_2+1/n_3}}
\bigg).
\end{align*}
\end{theorem}

This implies Theorem \ref{furedi}, as well as some more corollaries:
\begin{corollary}
\label{summarybm}
For all $n\geq 2$:
\begin{enumerate}
\item \label{n2-n2} $\bm(n,n,n^2-n/2) ~\geq~ n$
when $n$ is even.
\item $\bm (n,n,{n\choose 2}) ~\geq~ n-1$.
\item $\bm(n,n,2n-1) ~\geq~
\max\left(
\ceil{\frac{2 n - 1}{3}},
\ceil{\frac{3 n }{5}}
\right)$.
\item $\bm(n, 2n-1,2n-1) ~\geq~ n$.
\item $\bm(k,n,n) ~\geq~ \min(k,\ceil{\frac{n}{2}})$ for all $k\geq 1$.
\end{enumerate}
\end{corollary}
The proof of Theorem \ref{thm:tri-lower} uses topological tools which are presented in Section \ref{sec:tools}.
The proof itself is presented in Section \ref{sec:tripartite}.

Section \ref{sec:upperbound-bm} provides upper bounds on $\bm$,
in which two of the parts have equal size $n$ and the third is of a different size $k$. When $2\leq k \leq n$, we have:
\begin{theorem}\label{thm:tri-upper-knn}
\begin{align*}
\bm(k,n,n)\le 
\begin{cases}
\min(k,\floor{\frac{k}{2}+\frac{n}{4}}) & \text{If $n$ is even;}
\\
\min(k,\floor{\frac{k}{2}+\frac{n+3}{4}}) & \text{If $n$ is odd;}
\\
\min(k,\ceil{\frac{n}{2}}) & \text{If $k-\floor{\frac{n}{2}}$ divides $\floor{\frac{n}{2}}$.}
\end{cases}
\end{align*}
\end{theorem}

\begin{theorem}\label{thm:main_negative}

 If $r\geq 1$ and $n\leq k\leq rn$  then:

\begin{align*}
\bm(n,n,k) \leq \left\lceil 2r n/(2r+1) \right\rceil.
\end{align*}
\end{theorem}

Putting $r=\frac{n-1}{2}$ and $r=2$, respectively, yields: 
\begin{corollary}\hfill
\label{cor:bm-upper-bound}
\begin{enumerate}
\item 
\label{item:nchoose2}
$\bm(n,n, \binom{n}{2}) \leq n-1$, 
\HiddenExplanation{
	$r = (n-1)/2;
	rn = {n\choose 2} --- integer;
	n'  = \frac{n(n-1)}{n} = n-1 --- integer.
	$
}
\item 
\label{item:2n-1}
$\bm(n,n,2n) \leq \ceil{4n/5}$, and 
$\bm(n,n,2n-1) \leq \ceil{4n/5}$.
\HiddenExplanation{
	$r = (2n-1)/n;
	n'  = \frac{4n-2}{5n-2} n \leq \frac{4}{5}n$.
}
\end{enumerate}
\end{corollary}
%
It would be  of interest to close  the gap between the upper bound of $\sim 4n/5$ from Corollary \ref{cor:bm-upper-bound}(2) and the lower bound of $\sim 2n/3 $ from Corollary \ref{summarybm}(3). 

A more general open problem is to find nontrivial upper bounds on $\bm$ when all 3 arguments are distinct.

It is unlikely that $\bm$ is monotone in general (see Section \ref{sec:monotonicity}), and it is unclear whether any  of the
above bounds is tight.
We conjecture that  Corollary \ref{cor:bm-upper-bound}(2) is tight:
\begin{conjecture}
If $n \geq 4$ then $\bm(n,n,\binom{n}{2}+1) \geq n$.
\end{conjecture}


Section \ref{sec:higher-dimensions} presents an extension to $d$-partite hypergraphs for $d > 3$.

Our original  motivation for studying the $\bm$ function was proving results on cake division. We shall prove that if $\bm(n_1, n_2, \ldots ,n_d) \geq m$ then, when $d$ cakes are cut into respectively $n_t$ parts, $t\le d$, at least $m$ ``players'' can be allocated multiple parts of cakes fulfilling their pre-fixed requirements (these notions will be defined in Section \ref{sec:cakes}). 
These are re-formulations of  problems about multidimensional KKM theorems, which are described in Section \ref{sec:kkm}.

\section{Topological tools}
\label{sec:tools}

A hypergraph $\C$ is called a {\em simplicial complex} if  it is closed down, namely,
$f \subseteq e$ and $e \in \C$ imply $f \in \C$.
 For a subset $X$ of $V(\C)$ let $\C[X]=\{e \in \C \mid e \subseteq X\}$. 
  $\cc$ is called {\em homologically 
$k$-connected}  if for every $-1 \le j \le k$, the $j$-th reduced simplicial homology group of $\cc$ with rational coefficients $\tilde{H}_j(\cc)$ vanishes.

There is also a homotopic version: 
$\cc$ is called {\em homotopically 
$k$-connected}  if for every $-1 \le j \le k$, every continuous function 
$f: ||C|| \to S^j$ can be extended to a function $\tilde{f} \to B^{j+1}$.

%
%

The  \emph{homological (resp. homotopic)
connectivity} $\eta(\cc)$ (resp. $\eta_h(\cc)$)
of $\cc$ is the largest $k$ for which $\cc$ is homologically (resp. homotopically)
$k$-connected, plus $2$.%
\footnote{
$\tilde{H}_{-1}(\cc)$ is defined as the trivial group.
Therefore, if $\tilde{H}_{0}(\cc)$ is already non-trivial, then 
$\eta(\cc) = (-1)+2 = 1$,
which is its smallest possible value.
}

It is known that $\eta \ge \eta_h$, with equality if $\eta_h \ge 3$ (namely if the complex is simply connected).  
This follows from a theorem of Witold Hurewicz \citep{hatcher2002algebraic}[p.366,Thm.4.32].

Given sets $\cv := (V_1,\ldots,V_n)$, a \emph{$\cv$-transversal} is a function $f:\cv\to \cup_{i=1}^n V_i$ such that $f(V_i)\in V_i$ for all $i$.
Given a simplicial complex $\cc$ with $V(\cc)=\cup_{i=1}^n V_i$, a 
\emph{$\cc$-$\cv$-transversal} is a $\cv$-transversal $f$ whose image 
$f(\cv)$ 
is an element of $\cc$.
The theorem below gives a sufficient condition for existence of a $\cc$-$\cv$-transversal.
For a set $K \subseteq \{1,\ldots,n\}$, we denote $V_K:=\bigcup_{i\in K}V_i$.

\begin{theorem}[Topological Hall Theorem]\label{ah}
Let $\cc$ be simplicial complex and $\cv := \{V_1,\ldots,V_n\}$ subsets of $V(\cc)$.
If $\eta(\cc[V_K]) \ge |K|$ for every $K \subseteq
[n]$, then there exists 
a $\cc$-$\cv$-transversal.
\end{theorem}
The topological Hall theorem was essentially proved in \cite{ah}.
In the above form it was noted by the first author, as attributed in \cite{me2}.

A standard argument of adding dummy vertices yields a  deficiency version of Theorem \ref{ah}.

\begin{theorem}\label{deficiency}
Let $\cc$ be a simplicial complex, $\cv := \{V_1,\ldots,V_n\}$ subsets of $V(\cc)$, and $d\geq 0$ an integer.
If $\eta(\cc[V_K]) \ge |K|-d$ for every $K \subseteq [n]$,
then there exists a $\cc$-$\cv'$-transversal for some subset $\cv'\subseteq \cv$ with  $|\cv'| \geq n-d$.
\end{theorem}

\HiddenExplanation{
NOTE: The above theorems *cannot* guarantee an injective function.
To see this, let $\cc$ be the complete simplicial complex (= a ball) on $n$ vertices. In this case, the connectivity of $\cc[V_K]$ for all $K$ is infinite, so both theorems imply the existence of a $\cc$-$\cv$ transversal, for any $\cv$. 
But if $|\cv| > |V(\cc)|$, the transversal obviously cannot be injective.
}

We shall apply these theorems to independence complexes of graphs. The independence complex $\ci(G)$ of a graph $G$ consists of the independent sets in $V(G)$ (a set is independent if it does not contain an edge of $G$). The {\em line graph} $L(G)$ of a graph $G$ has $E(G)$ as its vertex set, and two edges in $E(G)$ form an edge in $L(G)$ if they intersect. Clearly, an independent set in $L(G)$ is a matching in $G$, so $\ci(L(G))$ is the matching complex of $G$, usually denoted $\cm(G)$.

A hypergraph $H$ is called {\em bipartite} with sides $X,Y$ (playing asymmetric roles) if $V(H)=X\cup Y$, $X \cap Y=\emptyset$ and  $|e\cap X|=1$ for all $e\in H$.%
\footnote{
Note that a bipartite hypergraph is not the same as a $d$-partite hypergraph with $d=2$ (every $d$-partite hypergraph is bipartite, but the opposite is not necessarily true).
}
We shall apply Theorems \ref{ah} and \ref{deficiency} to such hypergraphs. 
For every $x \in X$ let $N_H(x)$ be the neighborhood of $x$ in $Y$, namely $\{f\subseteq Y \mid f\cup\{x\} \in H\}$. 
For a subset $K$ of $X$, let $N_H(K) := \biguplus_{x \in K}N(x)$.
The $\biguplus$ means that 
we treat $N_H(K)$ as a multi-set (and a multi-hypergraph): identical neighbors of two elements of $K$ induce two elements in $N_H(K)$. 
If $H$ is a $d$-partite hypergraph, then $N_H(K)$ is a $(d-1)$-partite multi-hypergraph.

Applied to this setting, Theorem \ref{deficiency} yields:
\begin{corollary}
\label{cor:topological-hall}
Let $H$ be a bipartite hypergraph with sides $X,Y$ and $d\geq 0$ an integer.
 If $\eta(\cm(N_H(K))) \ge |K|-d$ for every $K \subseteq
X$, then $H$ has a matching of size $|X|-d$.
\end{corollary}

\HiddenExplanation{
Here is how the corollary follows from the theorem.
\begin{itemize}
\item The complex $\cc$ is $\cm(N_H(X)) = \ci(L(N_H(X)))$.
\item The vertex set $V(\cc)$ is $N_H(X)$. 
\item The subsets in $\cv$ are $V_x := N_H(x)$ for each $x\in X$
(they are not necessarily disjoint --- different vertices of $X$ may have the same neighbors).
\item 
For any $K\subseteq X$ we have $\cc[V_K] = \ci(L(N_H(K))) = \cm(N_H(K))$. So the condition in the corollary statement is equivalent to that in Theorem \ref{deficiency}. 
\item 
Therefore, there exists
a $\cc$-$\cv'$-transversal for some subset $\cv'\subseteq \cv$ with  $|\cv'| \geq |X|-d$;
denote the function by $f$. 
\item 
Now, we construct a matching in $X$ as follows:
$M := \{ \{x\} \cup f(V_x) ~|~ V_x\in \cv' \}$.
That is, for each $x\in X$ for which $V_x\in \cv'$, 
we take its image by $f$, and add to it the element $x$.
\item 
Note that $f(V_x)\in V_x = N_H(x)$, so $\{x\} \cup f(V_x)$ is indeed an edge in $H$.
\item 
Moreover, the image $f(\cv')$ is an element of $\cc = \cm(N_H(X))$, so all the $f(V_x)$ are pairwise-disjoint.
Therefore, $M$ is indeed a matching in $H$, and its size is at least $|X|-d$.
\end{itemize}
}

It is easier to work with the following slightly more general version:
\begin{corollary}
\label{cor:topological-hall-def}
Let $H$ be a bipartite hypergraph with sides $X,Y$ and $g: \mathbb{Z}_{\geq 0}\to \mathbb{Z}_{\geq 0}$ an integer function for which $g(z+1)\leq g(z)+1$ for all $z\in \mathbb{Z}_{\geq 0}$.
 If $\eta(\cm(N_H(K))) \ge g(|K|)$ for every $K \subseteq
X$, then $H$ has a matching of size $g(|X|)$.
\end{corollary}
\begin{proof}
The condition $g(z+1) \leq g(z)+1$ is equivalent to $z-g(z)$ being weakly-increasing on the integers. By assumption,
\begin{align*}
\eta(\cm(N_H(K))) \ge g(|K|) 
&= |K| - (|K|-g(|K|))
\\
&\geq |K| - (|X|-g(|X|)) && \text{since $|X|\geq |K|$.}
\end{align*}
By Corollary \ref{cor:topological-hall} there is a matching of size $|X| - (|X|-g(|X|)) = g(|X|)$.
\end{proof}

In order to apply Theorems~\ref{ah} --- 
\ref{cor:topological-hall-def}, one needs combinatorially formulated lower bounds on $\eta$, in particular on $\eta(\ci(G))$ for a graph $G$. 

A general lower bound on $\eta(\ci(G))$ is due to Meshulam \cite{me2}. Given an edge $e$ in a graph $G$ we denote by $G-e$ the graph obtained by removing $e$, and by $G \neg e$ the graph obtained by removing the vertices of $e$ and all their neighbors (together with the edges incident to them). 
The Meshulam bound is given by:

\begin{theorem}\label{meshnongame}
For every edge $e$ in a graph $G$ 
$$\eta(\ci(G)) \ge \min(\eta(\ci(G-e)), \eta(\ci(G\neg e))+1).$$
\end{theorem}


This bound  is conveniently expressed in terms of a
game between two agents, CON (wishing to prove high connectivity) and NON (the Mephistophelian ``spirit of perpetual negation''), on the graph $G$.  At each
step, CON chooses an edge $e$ from the graph remaining at this
stage, where in the first step the graph is $G$. NON can then either

\begin{enumerate}
\item delete $e$ from the graph (we call such a step a ``deletion'' or ``disconnection''),

or

\item remove from the graph the two endpoints of $e$, together with all neighbors of these vertices and the edges incident to them (we call such a step
an ``explosion'', and denote by $G\neg e$ the resulting graph).

\end{enumerate}

\noindent The result of the game (payoff to CON) is defined as follows: if at some point there
remains an isolated vertex $v$, the result is $\infty$ (the independence complex is then contractible to $v$, hence it is infinitely connected). Otherwise, at some
point all vertices have disappeared, in which case the result of the
game is the number of explosion steps. We define $\Psi(G)$ as the value of the game, i.e., the result obtained by optimal play on the graph $G$.

\begin{convention} \label{conventionmin} Henceforth we shall assume that NON always chooses the best strategy for him, namely he removes $e$ if   $\min(\Psi(G-e), \Psi(G\neg e)+1)=\Psi(G-e)$, and explodes it if $\min(\Psi(G-e), \Psi(G\neg e)+1)=\Psi(G\neg e)+1$. \end{convention}

Theorem \ref{meshnongame} can be stated as:

\begin{theorem}\label{etaPsi}
$\eta(\ci(G))\ge \Psi(G)$.
\end{theorem}

In fact, the stronger $\eta_h(\ci(G))\ge \Psi(G)$  is true \cite{aberger}.

\begin{remark}The game formulation first appeared in \cite{abz}. For an explicit proof of Theorem~\ref{etaPsi} using the recursive definition of $\Psi$, see Theorem~1 in \cite{adba}.
\end{remark}

Based on Theorem \ref{etaPsi}, in
Corollary \ref{cor:topological-hall-def}
 $\eta(\cm(N_H(K)))$ can be replaced by  $\Psi(L(N_H(K)))$.

\section{Matchings: from bipartite graphs to tripartite hypergraphs}
\label{sec:tripartite}
To prove Theorem \ref{thm:tri-lower}, we start from K\"onig's theorem for bipartite graphs \cite{konig1916graphok}, which says that $\bm(n,n)=n$ for all $n\geq 1$. We generalize it in several steps.

\subsection{From 1-1 matchings to many-to-many matchings}
K\"onig's theorem implies that, for any integers $n_2\geq n_1\geq 1$, $\bm(n_1, n_2) = n_1$.
This cannot be improved as long as we consider standard (1-to-1) matchings. 
In order to take advantage of the additional vertices in the $n_2$ side, we need  to consider many-to-many matchings.

\begin{definition}
Let $H$ be a $d$-partite hypergraph with sides $V_1,  \ldots ,V_d$.

(a)
Let $m_1,\ldots,m_d$ be positive integers. 
An \emph{$(m_1,\ldots,m_d)$-matching} in $H$
is an integral function $g: H\to \mathbb{Z}_{\geq 0}$ with the following property:
\begin{align*}
0\leq \deg_g(v) \leq m_t
&&
\text{for all $v\in V_t$ for all $t\in [d]$.}
\end{align*}

(b)
Let $r_1,\ldots,r_d$ be positive reals.
An \emph{$(r_1,\ldots,r_d)$-fractional-matching} in $H$
is a real function $f: H\to \mathbb{R}_{\geq 0}$ with the following property:
\begin{align*}
0\leq \deg_f(v) \leq r_t
&&
\text{for all $v\in V_t$ for all $t\in [d]$.}
\end{align*}
\end{definition}
An ordinary matching is a $(1,\ldots,1)$-matching, and an ordinary fractional matching is a $(1,\ldots,1)$-fractional-matching.

The following lemma generalizes K\"onig's theorem to many-to-many matchings in bipartite graphs. It reduces to K\"onig's theorem when $k=1$ and $n_1=n_2$.

\begin{lemma}
 \label{lem:fractional-implies-matching}
Let $G$ be a bipartite graph with sides $V_1, V_2$.
Let $n_1, n_2$ be integers with $n_2\geq n_1\geq 1$.
If there exists a 
$(1/n_1,1/n_2)$-fractional-matching
$f:E(G) \to \mathbb{R}_{\geq 0}$, then, for every $k\geq 1$, 
$G$ has:

(a) A $(\floor{k n_2/n_1},k)$-matching of size $\floor{\floor{k n_2/n_1} \cdot n_1 |f|}$, and ---

(b) A $(\ceil{k n_2/n_1},k)$-matching of size $ \floor {k \cdot n_2 |f|}$.
\end{lemma}

\begin{proof}
\HiddenExplanation{
	The previous proof used the matching polytope,
	and relied on the fact that the matching polytope of a bipartite graph is integral.
	It is not clear how to extend this proof to higher dimensions, since the matching polytope of such hypergraphs is not integral.
	The new proof seems simpler and easier to extend.
}

(a) For each vertex $v_1\in V_1$, construct $\floor{k n_2/n_1}$ clones (including $v_1$);
for each vertex $v_2\in V_2$, construct $k$ clones (including $v_2$). 
Add edges between each clone of $v_1$ and each clone of a neighbor $v_2$ of $v_1$.
Set the weight of each edge between a $v_1$ clone and a $v_2$ clone to
\begin{align*}
\frac{n_1}{k}\cdot  f((v_1,v_2)).
\end{align*}
With the new weight function,
the degree of each clone of $v_1\in V_1$ is at most $\frac{n_1}{k} \cdot (1/n_1) \cdot k = 1$, and the degree of each clone of 
$v_2\in V_2$ is $\frac{n_1}{k} \cdot (1/n_2) \cdot \floor{kn_2/n_1} \leq  1$, so it is a fractional matching.
Since each edge is cloned $k \floor{k n_2/n_1}$ times, the total size of this fractional matching is $\floor{k n_2/n_1}k\cdot \frac{n_1}{k}|f|$.
By K\"onig's theorem, the clone graph has a matching of the same size (rounded down). Re-combining the clones gives a $(\floor{k n_2/n_1},k)$-matching 
of size $\floor{\floor{k n_2/n_1} \cdot n_1 |f|}$.

(b) 
This part differs from (a) only when $k n_2/n_1$ is not an integer. In this case, add one more clone for each vertex $v_1\in V_1$. Link this new clone to each clone of a neighbor $v_2$ of $v_1$, and assign to each new edge a weight of
\begin{align*}
\left(n_2 - \frac{n_1}{k}\cdot \floor{k n_2/n_1}\right)\cdot f((v_1,v_2)).
\end{align*}
The degree of each new clone is at most $\left(n_2 - \frac{n_1}{k}\cdot \floor{k n_2/n_1}\right) \cdot (1/n_1) \cdot k = 
\frac{k n_2}{n_1} - \floor{\frac{k n_2}{n_1}} 
\leq 
1$.
The degree of each clone of $v_2$ is now at most
$
\frac{n_1}{k} \cdot (1/n_2) \cdot \floor{kn_2/n_1}
+
\left(n_2 - \frac{n_1}{k}\cdot \floor{k n_2/n_1}\right)
 \cdot (1/n_2)
= 1$.
Therefore the new function is still a fractional matching.
Its size is now 
$\floor{k n_2/n_1}k\cdot \frac{n_1}{k}|f|
+
\left(n_2 - \frac{n_1}{k}\cdot \floor{k n_2/n_1}\right)\cdot k \cdot |f| = k\cdot n_2|f|
$,
so there is an integral matching of the same size. 
Re-combining the clones gives a $(\ceil{k n_2/n_1},k)$-matching of size $ \floor {k \cdot n_2 |f|}$.
\end{proof}

Every $(n_1,n_2)$-fractionally balanced graph has, by definition, 
a $(1/n_1,1/n_2)$-fractional-matching of size $1$.
Therefore, by Lemma \ref{lem:fractional-implies-matching}:
\begin{corollary}
\label{cor:balanced-implies-matching}
For all $n_2\geq n_1\geq 1$,
every $(n_1,n_2)$-fractionally balanced graph has:
\begin{itemize}
\item An $(\floor{n_2/n_1},1)$-matching of size 
$\floor{n_2/n_1}\cdot n_1$;
\item An $(\ceil{n_2/n_1},1)$-matching of size 
$n_2$;
\item An $(\floor{2 n_2/n_1},2)$-matching of size 
$\floor{2 n_2/n_1}\cdot n_1$;
\item An $(\ceil{2 n_2/n_1},2)$-matching of size 
$2 n_2$.
\end{itemize}
\end{corollary}

\begin{remark}
Corollary \ref{cor:balanced-implies-matching} and Lemma \ref{lem:fractional-implies-matching}
are tight when $n_2/n_1$ is an integer: the disjoint union of $n_1$ copies of the star graph $K_{1,n_2/n_1}$ is $(n_1,n_2)$-fractionally balanced, and the largest $(kn_2/n_1,k)$-matching that fits into it has size $k n_2$ (when the weight of every edge is $k$).
\end{remark}

\subsection{From many-to-many matchings to homological connectivity}
The following lemma says that the existence of a many-to-many matching in a bipartite graph implies a lower bound on the homological connectivity of the matching complex.

\begin{lemma}
\label{lem:matching-implies-eta}
Let $G$ be a bipartite graph with sides $V_1, V_2$.
Let $m_1 \geq m_2 \geq 1$ be positive integers. 
If $G$ has an $(m_1,m_2)$-matching $g$, then 
\begin{align*}
\eta(\cm(G)) \ge
\ceil {\frac{|g|}{
m_1+m_2+\max(m_2-2,0)
}}.
\end{align*}
\end{lemma}
\begin{proof}
Since $\cm(G) =\ci(L(G))$, by Theorem \ref{etaPsi}
it is sufficient to prove that $\Psi(L(G))\geq \ceil {\frac{|g|}{m_1+m_2+\max(m_2-2,0)}}$.
This can be proved by playing Meshulam's game on $L(G)$. 

One can view $G$ as a $2$-dimensional array of cells. 
Each vertex $v_1\in V_1$ corresponds to a row and each vertex $v_2\in V_2$ corresponds to a column.
Each edge $(v_1,v_2)\in G$ corresponds to a cell in the intersection of row $v_1$ and column $v_2$.

In $L(G)$, the vertices are the cells, and each edge corresponds to a pair of cells in the same row or column.
Since $g$ is an $(m_1,m_2)$-matching,
the sum of $g$-weight in each row is at most $m_1$ and in each column at most $m_2$.

We show that, if CON offers pairs of cells in a specific order, then each explosion made by NON destroys cells with a total $g$-weight of at most $m_1+m_2 + \max(m_2,2)-2$. This implies that NON needs at least $\ceil {\frac{|g|}{m_1+m_2 + \max(m_2-2,0)}}$ explosions to destroy all edges.

CON starts by offering pairs of cells in the same row with a weight of at least $1$, that is, pairs of the form $(v_1,v_2')$ and $(v_1,v_2'')$ with $g(v_1,v_2')\geq 1$ and $g(v_1,v_2'')\geq 1$.
If NON explodes such a pair, then one row $v_1$ and two columns $v_2',v_2''$ are destroyed. 
Since a weight of at least $2$ is common to the row and columns, the total weight destroyed is at most $m_1 + 2m_2 - 2$.

If NON disconnects all such pairs, then CON goes on to offer all pairs of cells in the same column $(v_1',v_2)$ and $(v_1'',v_2)$ with $g(v_1',v_2)\geq 1$ and $g(v_1'',v_2) \geq 1$. Each such cell is now connected, in its row, only to cells with  weight $0$. Therefore, if NON explodes such a pair, then the total destroyed weight is the weight in column $v_2$, which is at most $m_2$.

If NON disconnects all offered pairs, then CON offers all pairs of cells in the same row $(v_1,v_2')$ and $(v_1,v_2'')$ with $g(v_1,v_2') \geq 1$ and $g(v_1,v_2'') = 0$.
The cell $(v_1,v_2')$ is connected, in its column, only to cells with  weight $0$. Therefore, an explosion destroys only the weight in row $v_1$ which is at most $m_1$, and the weight in column $v_2''$ which is at most $m_2$, for a total of $m_1+m_2$.

Finally, CON offers all remaining connected pairs of  cells. Each cell is now connected, in its row, only to cells with  weight $0$. Therefore, an explosion destroys only the weight in at most two columns, which is at most $2 m_2$.

In sum, each explosion destroys a total weight of at most
$\max(m_1+2m_2-2, m_2, m_1+m_2, 2m_2) = m_1+m_2+\max(m_2-2,0)$.
\end{proof}

For $m_2 \in \{1,2\}$, 
the right-hand side of Lemma \ref{lem:matching-implies-eta}
is $\ceil {\frac{|g|}{
m_1+m_2}}$.
The following proposition shows that
the bound it gives is tight.

\begin{proposition}
For all integers $m\geq 2$ and $d \geq 1$, there exists a bipartite graph $G$ with the following properties:

(a) an $(m, 1)$-matching $g_1$ with $|g_1|=(m+1)\cdot d$;

(b) an $(m,2)$-matching $g_2$ with $|g_2|=(m+2)\cdot d$;

(c) 
$\eta(\cm(G)) \leq d$.

\end{proposition}
\begin{proof}
Let $G_m$ be the following bipartite graph:

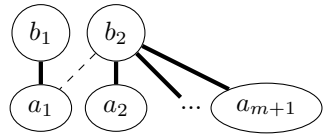
\begin{center}
\begin{tikzpicture}[scale=1]
\node[draw,ellipse] (b1) at (0,2) {$b_1$};
\node[draw,ellipse] (b2) at (1,2) {$b_2$};

\node[draw,ellipse] (a1) at (0,1) {$a_1$};
\node[draw,ellipse] (a2) at (1,1) {$a_2$};
\node (adots) at (2,1) {$...$};
\node[draw,ellipse] (a3) at (3,1) {$a_{m+1}$};

\draw[draw,ellipse,ultra thick] (b1) -- (a1);
\draw[draw,ellipse,dashed] (b2) -- (a1);
\draw[draw,ellipse,ultra thick] (b2) -- (a2);
\draw[draw,ellipse,ultra thick] (b2) -- (adots);
\draw[draw,ellipse,ultra thick] (b2) -- (a3);
\end{tikzpicture}
\end{center}
Let $e^*$ be the thin dashed edge.
Let $f$ be an edge-weight function that assigns a weight of $1/m$ to all edges except $e^*$. Then $f$ is a $(1, 1/m)$-fractional-matching and $|f|=\frac{m+1}{m}$. 
By Lemma \ref{lem:fractional-implies-matching}  (with $k=1, n_1=1,n_2=m$),
$G_m$ has a $(m,1)$-matching of size $m+1$;
one such matching contains all edges except $e^*$.
The same graph $G_m$ has an $(m,2)$-matching of size $m+2$: assign a weight of $1$ to all edges adjacent to $b_2$ except $e^*$, and a weight of $2$ to $(b_1,a_1)$.

The line-graph $L(G_m)$ can be presented as follows, where cells with a * correspond to vertices in $L(G_m)$, and two cells are adjacent iff they are in the same row or column:
\begin{center}
\begin{tabular}{c|c|c|c|c|}
& $a_1$ &$a_2$ &$\ldots$ &$a_{m+1}$ \\
\hline
$b_1$: & * &  &  &  \\
\hline
$b_2$: & * & * & $\ldots$ & * \\
\hline
\end{tabular}
\end{center}

\HiddenExplanation{
	NON can destroy $L(G_m)$ with a single explosion using the following strategy:
	If CON offers a pair that does not contain the cell $e^*$; disconnect it; otherwise, explode it. Since $e^*$ is connected to all other cells, a single explosion destroys $G$.
	Therefore, $\Psi(L(G_m))=1$. 
}
Since $e^*$ is adjacent to all other cells, it is a singleton in $I(L(G_m))$, so $I(L(G_m))$ is disconnected. Hence,
$\tilde{H}_{0}(\cm(G_m))$ is non-trivial and
$\eta(\cm(G_m))=1$.

Now, for every $d\geq 1$, let $G_m^d$ be a disjoint union of $d$ copies of $G_m$. This graph has a $(1,1/m)$-fractional-matching of size $d\cdot |f|$, an $(m,1)$-matching of size $d\cdot (m+1)$, and an $(m,2)$-matching of size $d\cdot (m+2)$.
\HiddenExplanation{
	Since Meshulam's game is played separately on each connected component, $\Psi(L(G_m^d))=d\cdot \Psi(L(G_m)) = d$. 
	\erel{Does this ``additivity'' over connected components hold for $\eta$ too?
	Is the proposition true for 
	$\eta(\cm(G_m^d))$ too?
	Ron: $\eta$ is additive.}
}
The function $\eta$ is additive over connected components (this is a consequence of the K\"unneth formula from algebraic topology---see e.g.\ Section 3 of \cite{km}).
Hence, $\eta(\cm(G_m^d))=d\cdot \eta(\cm(G_m)) = d$. 
\end{proof}

As we have seen, bounds on $\bm$ can be achieved through bounds on the connectivity of matching complexes. It is worthwhile studying such bounds on their own.
\begin{notation}
For positive integers $n_1,n_2$ let  $\zeta(n_1,n_2)$ (resp. $\zeta_h(n_1,n_2)$) be the minimum of $ \eta(\cm(G))$ (resp. $ \eta_h(\cm(G))$)  over all  $(n_1,n_2)$-fractionally balanced bipartite graphs $G$. 
\end{notation}

Applying Lemma \ref{lem:matching-implies-eta} 
with $m_2=2$
to the last two items in Corollary \ref{cor:balanced-implies-matching} implies
 \begin{align*}
\zeta(n_1, n_2) \ge \ceil{\max\left(  \frac{\floor{2 n_2/n_1} n_1}{\floor{2 n_2/n_1} +2},
\frac{2 n_2}{\ceil{2 n_2/n_1} +2} \right)}.
\end{align*}


In particular:

 \begin{align*}
\zeta(n,~ n\cdot (n-1/2)) &~\ge~ n
\\
\zeta(n,~ n\cdot (n-1)/2) &~\ge~ n-1
\end{align*}

Below we prove an almost matching lower bound:
\begin{proposition}
$\zeta_h(n,(n-1)^2)<n$.
\end{proposition}

\begin{proof}

 
Let $m=n-1$, $N=m^2-m$,  $y = \frac{m+1}{m+N}, ~~x=\frac{1}{m+N}$.

Let $A=\{a_1,\ldots,a_m,a\}$ and 
$B=\{b_1,\ldots,b_{N+m}\}$. Define a function $f$ as follows.  

\begin{enumerate}
\item 
$f((a_i,b_j))=y$ for $i \in [m], j \in [N]$ [thin lines];
\item 
$f((a_i,b_j))=x$ for $i \in [m], j >N$ [dashed lines];
\item 
$f((a,b_j))=1$ for $j >N$ [thick lines].
\end{enumerate}

Let  $G$ be the bipartite graph  with respective sides $A$ and $B$, and edge set $\supp(f)$.

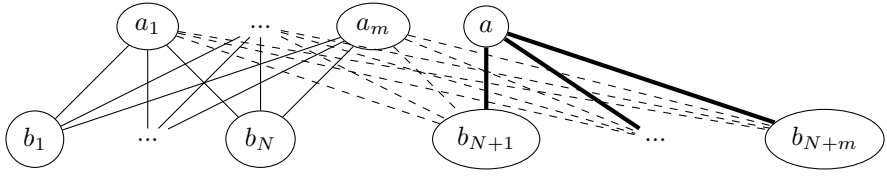
\begin{center}
\begin{tikzpicture}[scale=1.5]
\node[draw,ellipse] (a1) at (1,1) {$a_1$};
\node (adots) at (2,1) {$...$};
\node[draw,ellipse] (am) at (3,1) {$a_{m}$};
\node[draw,ellipse] (aa) at (4,1) {$a$};

\node[draw,ellipse] (b1) at (0,0) {$b_1$};
\node (bdots1) at (1,0) {$...$};
\node[draw,ellipse] (bN) at (2,0) {$b_N$};
\node[draw,ellipse] (bN1) at (4,0) {$b_{N+1}$};
\node (bdots2) at (5.5,0) {$...$};
\node[draw,ellipse] (bNm) at (7,0) {$b_{N+m}$};

\draw[ultra thick] (aa)  -- (bN1);
\draw[ultra thick] (aa) -- (bdots2);
\draw[ultra thick] (aa) -- (bNm);

\draw[] (a1)  -- (b1);
\draw[] (a1)  -- (bdots1);
\draw[] (a1)  -- (bN);
\draw[] (adots)  -- (b1);
\draw[] (adots)  -- (bdots1);
\draw[] (adots)  -- (bN);
\draw[] (am)  -- (b1);
\draw[] (am)  -- (bdots1);
\draw[] (am)  -- (bN);

\draw[dashed] (a1)  -- (bN1);
\draw[dashed] (a1)  -- (bdots2);
\draw[dashed] (a1)  -- (bNm);
\draw[dashed] (adots)  -- (bN1);
\draw[dashed] (adots)  -- (bdots2);
\draw[dashed] (adots)  -- (bNm);
\draw[dashed] (am)  -- (bN1);
\draw[dashed] (am)  -- (bdots2);
\draw[dashed] (am)  -- (bNm);
\end{tikzpicture}
\end{center}

$G$ is $(n, (n-1)^2)$ fractionally balanced,
since $\deg_f(a_i) = 
m$ for all $a_i\in A$ 
and   
$\deg_f(b_j) = 
(m+1)/m$ for all $b_j\in B$.
We claim that $\eta(\cm(G))\le m$ (in fact, equality holds, but we do not need this.) 

Let $X$ be the sub-complex of $\cm(G)$ 
induced by the following set of $m$ pairs of vertices:
$\{(a_i,b_i); (a_i,b_{N+i}) \mid i \le m\}$.
Each vertex of $X$ is adjacent in $\cm(G)$ to (=appears in the same matching as) all vertices of $X$ except its counterpart in the pair.
Therefore, $X$ is 
isomorphic to the boundary complex 
of the $m$-dimensional cross-polytope.

The face $M=\{(a_i,b_{N+i}) \mid i \le m\}$  of $X$ is a maximal face in $\cm(G)$,
so it is not contained in any $m$-dimensional simplex of $\cm(G)$. This means that $X$ cannot be filled in $\cm(G)$, in the sense of extending the function embedding the sphere
in the complex to 
the ball (we skip here a passage from 
piecewise linear functions to general functions)
proving that $\eta_h(\cm(G))\le m$. 
\end{proof}

\subsection{From homological connectivity to matchings in tripartite graphs}

Lemma \ref{lem:matching-implies-eta} and the topological Hall theorem provide together a general lower bound on $\bm$ for tripartite hypergraphs.

\begin{theorem*}[Theorem \ref{thm:tri-lower}]
For every
positive integers $n_1\leq n_2\leq n_3$:
\begin{align*}
(a)&&
\bm
(n_1, n_2, n_3)\geq
\min\bigg(
n_1,~
\ceil{\frac{\floor{2 n_3/n_2} n_2 }{\floor{2 n_3/n_2}+2}}
\bigg)
\\
(b)&&
\bm
(n_1, n_2, n_3)\geq 
\min\bigg(
n_1,~
\ceil{\frac{2 n_3}{\ceil{2 n_3/n_2}+2}}
\bigg)
\end{align*}
\end{theorem*}

\begin{remark}
Theorem \ref{thm:tri-lower} has a nicer form when $n_3/n_2$ is an integer:
\begin{align*}
\bm
(n_1, n_2, n_3)
&\geq
\min\bigg(
n_1,~
\ceil{\frac{1}{1/n_2+1/n_3}}
\bigg).
\end{align*}
Note that the most ``efficient'' triplets for this expression are the ones with $1/n_1 = 1/n_2+1/n_3$, e.g. $(n, 2n, 2n)$ or $(2n, 3n, 6n)$. 
This also indicates that the assumption $n_1\leq n_2\leq n_3$ is without loss of generality: any other selection would result in a smaller matching size.
\end{remark}

\begin{proof}[Proof of Theorem \ref{thm:tri-lower}]
Let $H$ be a fractionally balanced  hypergraph with sides $V_1, V_2, V_3$ with $|V_t|=n_t$ for $t\in[3]$.
Let $f$ be a corresponding weight function on $H$, normalized such that $|f|=1$ and  $\deg_f(v)=1/n_t$ for all $v\in V_t$.

For any $K\subseteq V_1$,
the neighbor set $N_H(K)$ is a bipartite graph contained in $V_2\times V_3$.
Let $f_K$ be 
the fractional matching induced by $f$ on $N_H(K)$.
This $f_K$ satisfies the conditions of Lemma \ref{lem:fractional-implies-matching} with $|f_K|=|K|/n_1$. We apply the lemma with $k=2$. 
Below, we denote $s := n_2/n_1$ and $r := n_3/n_2$.

By part (a) of Lemma \ref{lem:fractional-implies-matching}, $N_H(K)$ has a $(\floor{2 r},2)$-matching of size $\floor{\floor{2 r} \cdot s |K|}$, so by Lemma \ref{lem:matching-implies-eta}:
\begin{align*}
\eta(\cm(N_H(K))) 
&\ge
\ceil{\frac{\floor{2 r} s |K|}{\floor{2 r}+2}}
\end{align*}
If $\floor{2 r}s\leq \floor{2 r}+2$, then 
$g(z)=\ceil{\frac{\floor{2 r} s z}{\floor{2 r} +2}}$  satisfies the requirements of 
Corollary \ref{cor:topological-hall-def}, which implies that $H$ has a matching of size $\ceil{\frac{\floor{2 r} s n_1}{\floor{2 r}+2}} = 
\ceil{\frac{\floor{2 n_3/n_2}  n_2}{\floor{2 n_3/n_2}+2}}
$.
Otherwise, the above inequality implies that 
$\eta(\cm(N_H(K))) \geq|K|$,
so Corollary \ref{cor:topological-hall-def} with $g(z)=z$ implies that 
$H$ has a matching of size $n_1$. 
This proves part (a) of the theorem.

Similarly, by part (b) of Lemma \ref{lem:fractional-implies-matching}, $N_H(K)$ has a $(\ceil{2 r},2)$-matching of size $\floor{2 r \cdot s |K|}$, so by Lemma \ref{lem:matching-implies-eta}:
\begin{align*}
\eta(\cm(N_H(K))) 
&\ge
\ceil{\frac{\floor{2 r s |K|}}{\ceil{2 r}+2}}
\end{align*}
If $\ceil{2 r s} \leq \ceil{2 r} + 2$,
then $g(z) = \ceil{\frac{\floor{2 r s z}}{\ceil{2 r}+2}}$
satisfies the requirements of 
Corollary \ref{cor:topological-hall-def}, which
implies that $H$ has a matching of size at least $\ceil{\frac{\floor{2 r s n}}{\ceil{2 r}+2}} =  \ceil{\frac{2 n_3}{\ceil{2 n_3/n_2} +2}}$. 
Otherwise, 
$\ceil{2 r s} \geq \ceil{2 r} + 3$, 
so 
$2 r s \geq \ceil{2 r} + 2$.
Then 
the above inequality implies that 
$\eta(\cm(N_H(K))) \geq|K|$,
so Corollary \ref{cor:topological-hall-def} with $g(z)=z$ implies that 
$H$ has a matching of size $n_1$. 
This proves part (b) of the theorem.
\end{proof}


\begin{proof}[Proof of Corollary \ref{summarybm}]
~
\begin{enumerate}
\item 
$\bm(n,n,n^2-n/2) \geq n$: Apply Theorem \ref{thm:tri-lower}(a): $\floor {2 n_3/n_2} = 2n-1$ and $\ceil{\frac{\floor{2 n_3/n_2} n_2 }{\floor{2 n_3/n_2}+2}} = \ceil{\frac{n  (2n-1)}{2n+1} } = \ceil{n - \frac{2n}{2n+1}} = n$. 
\HiddenExplanation{
	$n-1/2$ is the largest value of $r$ for which the outcome is $n$; see 
	\url{https://math.stackexchange.com/a/3811723/29780}.
}

\item 
$\bm(n,n,{n\choose 2}) \geq  n-1$: Apply Theorem \ref{thm:tri-lower}(a):
$\floor {2 n_3/n_2} = n-1$,
and $\ceil{\frac{\floor{2 n_3/n_2} n_2 }{\floor{2 n_3/n_2}+2}} = 
\ceil{\frac{(n-1)n}{n+1} } = n-1$.

\item 
$\bm(n,n,2n-1) \geq  
\max\left(
\ceil{\frac{2 n - 1}{3}},
\ceil{\frac{3 n }{5}}
\right)$. Apply Theorem \ref{thm:tri-lower}: here
$n_3/n_2 =  2 - 1/n$.
Then part (a) gives
$\ceil{\frac{\floor{4-2/n} n }{\floor{4-2/n}+2}} = 
\ceil{\frac{3 n }{5}}$,
and part (b) gives 
$\ceil{\frac{4n-2}{\ceil{4-2/n}+2}} \geq
\ceil{\frac{4n-2}{6}} = 
\ceil{\frac{2n-1}{3}}$.

\item 
$\bm(n, 2n-1,2n-1) \geq  n$: apply Theorem \ref{thm:tri-lower}(b). Since
$n_3/n_2 =  1$, we have
$\ceil{\frac{2n_3}{\ceil{2 n_3/n_2}+2}} = \ceil{\frac{4n-2}{4}} =  n$.

\item 
$\bm(k,n,n) \geq  \min(k,\ceil{\frac{n}{2}})$:  apply Theorem \ref{thm:tri-lower}(b). Since  
$n_3/n_2 =  1$, we have
$\ceil{\frac{2n_3}{\ceil{2 n_3/n_2}+2}} = \ceil{\frac{2n}{4}} =  \ceil{\frac{n}{2}}$.
\end{enumerate}
\end{proof}

\section{Upper bounds on BM for tripartite hypergraphs}
\label{sec:upperbound-bm}
The Pasch hypergraph (defined in the introduction)
is $(2,2,2)$ fractionally balanced and its maximum matching size is $1$. 
Combined with Theorem \ref{thm:tri-lower} for $n_1=2$, it implies that, for $n_3\geq n_2\geq 2$:
\begin{align*}
\bm(2, n_2, n_3) =
\begin{cases}
1 & \text{when $n_2=n_3=2$;}
\\
2 & \text{otherwise (i.e. when $n_3>2$).}
\end{cases}
\end{align*}
Taking $\frac{n}{2}$ vertex-disjoint copies of the Pasch hypergraph when $n$ is even, or $\floor{\frac{n}{2}}$ copies plus one isolated edge when $n$ is odd,
yields an $(n,n,n)$-fractionally-balanced hypergraph in which the maximum matching size is $\ceil{n/2}$. So we have the following upper bound, which is tight by \citet{furedi1981maximum}:
\begin{theorem}
\label{thm:nnn}
For all $n\geq 1$,
\begin{align*}
\bm(n,n,n) \leq \ceil{\frac{n}{2}}.
\end{align*}
\end{theorem}
Proving an upper bound for sides of different sizes is more challenging. One could try to prove $\bm(k,n,n) \leq \ceil{\frac{n}{2}}$ for  $k<n$ (which would be tight by Theorem \ref{thm:tri-lower}) by deleting a vertex from $n-k$ copies of the Pasch graph. However, while each copy on its own would remain fractionally-balanced, the union of all copies would not. The theorem below provides a different construction that proves a weaker upper bound.

\begin{theorem*}[Theorem \ref{thm:tri-upper-knn}]
For all  $n\geq k\geq 2$,
\begin{align*}
\bm(k,n,n)\le 
\begin{cases}
\min(k,\floor{\frac{k}{2}+\frac{n}{4}}) & \text{If $n$ is even;}
\\
\min(k,\floor{\frac{k}{2}+\frac{n+3}{4}}) & \text{If $n$ is odd;}
\\
\min(k,\ceil{\frac{n}{2}}) & \text{If $k-\floor{\frac{n}{2}}$ divides $\floor{\frac{n}{2}}$.}
\end{cases}
\end{align*}
\end{theorem*}

\begin{proof}
Let $m :=\floor{ \frac{n}{2}}$. When $k\leq m$, the right-hand side equals $k$ and the theorem is trivial, so assume $k>m$.
Let $I=\{1,2,\dots, m\}$ and  $J=\{ m +1,m+2,\dots,k\}$.
Define the following hypergraphs on vertex set $V := \{a_1,\ldots,a_k,b_1,\ldots,b_n,c_1,\ldots,c_n\}$:
\begin{itemize}
    \item $H_1 = \{a_i b_{2i-1} c_{2i} \mid i\in I\} \cup \{a_i b_{2i} c_{2i-1} \mid i\in I\}$,
    \item $H_2 = \{a_j b_i c_i \mid j\in J, i\in [2 m] \}$,
    \item $H_3=\{a_j b_n c_n \mid j\in J\}$.
\end{itemize}

\textbf{(a)} If $n$ is even, let $H=H_1\cup H_2$. The diagram below shows some edges of $H_1$ (thick) and $H_2$ (thin). Assign weights $\frac{1}{2}$ to every edge in $H_1$ and $\frac{1}{n}$ to every edge in $H_2$.
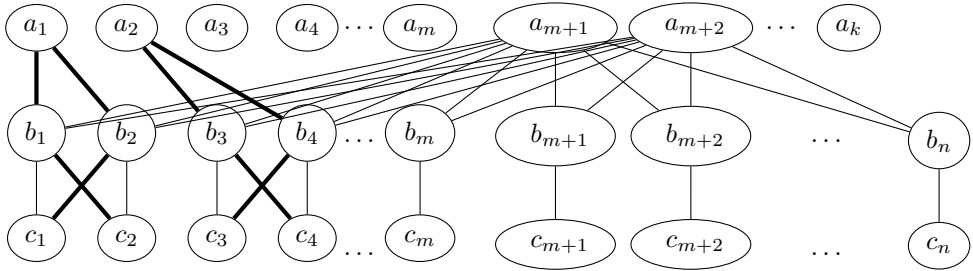
\begin{center}
\def\aline{1}
\def\bline{0.5}
\def\cline{0}
\def\height{0.7}
\begin{tikzpicture}[scale=3]
\node[draw,ellipse] (a1) at (0,\aline) {$a_1$};
\node[draw,ellipse] (a2) at (0.4,\aline) {$a_2$};
\node[draw,ellipse] (a3) at (0.8,\aline) {$a_3$};
\node[draw,ellipse] (a4) at (1.2,\aline) {$a_4$};
\node  at (1.43,\aline) {$\ldots$};
\node[draw,ellipse] (am) at (1.7,\aline) {$a_m$};
\node[draw,ellipse] (am1) at (2.3,\aline) {$~a_{m+1}$};
\node[draw,ellipse] (am2) at (2.9,\aline) {$~a_{m+2}$};
\node  at (3.3,\aline) {$\ldots$};
\node[draw,ellipse] (ak) at (3.6,\aline) {$a_k$};
\node[draw,ellipse,below=\height of a1] (b1)  {$b_1$};
\node[draw,ellipse,below=\height of a2] (b2)  {$b_2$};
\node[draw,ellipse,below=\height of a3] (b3)  {$b_3$};
\node[draw,ellipse,below=\height of a4] (b4)  {$b_4$};
\node  at (1.43,\bline) {$\ldots$};
\node[draw,ellipse,below=\height of am] (bm)  {$b_m$};
\node[draw,ellipse,below=\height of am1] (bm1)  {$~b_{m+1}$};
\node[draw,ellipse,below=\height of am2] (bm2) {$~b_{m+2}$};
\node  at (3.5,\bline) {$\ldots$};
\node[draw,ellipse] (bn) at (4,\bline) {$b_n$};
\node[draw,ellipse,below=\height of b1] (c1) {$c_1$};
\node[draw,ellipse,below=\height of b2] (c2)  {$c_2$};
\node[draw,ellipse,below=\height of b3] (c3)  {$c_3$};
\node[draw,ellipse,below=\height of b4] (c4)  {$c_4$};
\node  at (1.43,\cline) {$\ldots$};
\node[draw,ellipse,below=\height of bm] (cm){$c_m$};
\node[draw,ellipse,below=\height of bm1] (cm1) {$~c_{m+1}$};
\node[draw,ellipse,below=\height of bm2] (cm2) {$~c_{m+2}$};
\node  at (3.5,\cline) {$\ldots$};
\node[draw,ellipse,below=\height of bn] (cn) {$c_n$};
\draw [ultra thick] (a1) -- (b1) -- (c2);
\draw [ultra thick] (a1) -- (b2) -- (c1);
\draw [ultra thick] (a2) -- (b3) -- (c4);
\draw [ultra thick] (a2) -- (b4) -- (c3);

\draw (am1) -- (b1);
\draw (am1) -- (b2);
\draw (am1) -- (b3);
\draw (am1) -- (b4);
\draw (am1) -- (bm);
\draw (am1) -- (bm1);
\draw (am1) -- (bm2);
\draw (am1) -- (bn);

\draw (am2) -- (b1);
\draw (am2) -- (b2);
\draw (am2) -- (b3);
\draw (am2) -- (b4);
\draw (am2) -- (bm);
\draw (am2) -- (bm1);
\draw (am2) -- (bm2);
\draw (am2) -- (bn);

\draw (b1) -- (c1);
\draw (b2) -- (c2);
\draw (b3) -- (c3);
\draw (b4) -- (c4);
\draw (bm) -- (cm);
\draw (bm1) -- (cm1);
\draw (bm2) -- (cm2);
\draw (bn) -- (cn);
\end{tikzpicture}
\end{center}
Then $\deg a_i = 1$ and $\deg b_i = \deg c_i = k/n$, so $H$ is fractionally balanced.
\HiddenExplanation{
	$\deg a_i = 1/2+1/2 = 1,
	\deg a_j = m/n+m/n = 1;
	\deg b_i = 1/2+(k-m)/n = 1/2 + k/n-1/2 = k/n,
	\deg b_{m+i} = 1/2+(k-m)/n = k/n;
	\deg c_i = \deg b_i = k/n.
	$
}
Let $M$ be any matching in $H$, and assume that some $x\geq 0$ edges of $M$ come from $H_1$. Each such edge touches two $(b_i,c_i)$ edges in the graph induced by $H_2$ on $B\times C$. Hence, $M$ can contain at most $\min\{k-m,2m-2x\}$ edges of $H_2$. 
\HiddenExplanation{
	At most $k-m$ --- since $H_2$ edges contains only vertices $a_j$ with $j\in J$, and $|J|=k-m$.
	At most $2m-2x$ --- since each edge in $H_1$ ``blocks'' two pairs of $B\times C$ in $H_2$.
}
So (since $m=n/2$):
$$\nu(H) \le x+\min\{k-m, 2m-2x \}=\min\{k-m+x,2m-x\} \le \frac{k+m}{2}
= \frac{k}{2} + \frac{n}{4} .
$$  

\textbf{(b)}
If $n$ is odd,  let $H=\bigcup_{i=1}^3
H_i$.  To see that $H$ is fractionally balanced, assign weights $\frac{1}{2}$ to every edge in $H_1$,  $\frac{1}{n-1} - \frac{k}{n(n-1)(k-m)}$ to every edge in $H_2$, 
and    $\frac{k}{n(k-m)}$ to every edge of $H_3$,
to get again $\deg a_i = 1$ and $\deg b_i = \deg c_i = k/n$.
\HiddenExplanation{
	In the odd case, $2m = n-1$.
	$\deg a_i = 1/2+1/2 = 1;
	\deg a_j = 2m/(n-1) - 2mk/(n(n-1)(k-m)) + k/(n(k-m)) = 1 -k/n(k-m)+k/n(k-m) = 1;
	\deg b_i = 	\deg c_i = 1/2 + (k-m)/(n-1) - k/n(n-1) = 1/2 + (k-m)/2m - k/2mn = k/2m - k/2mn = k(n-1)/2mn = k/n;
	\deg b_n = \deg c_n = (k-m) \cdot k/n(k-m)  = k/n
	$
}
$H_3$ adds at most one edge to the maximum matching.
So (since $m=(n-1)/2$): $$\nu(H) \le \frac{k+m}{2} + 1
= \frac{k}{2} + \frac{n+3}{4} .
$$  

\textbf{(c)}
Write $k'=k-m$. Since $k-m$ divides $m$ we have  $m=qk'$ for some integer $q\ge 1$. 
The idea is to replace the hypergraph $H_2$, in which each $a_j$ is connected to all $2 m$ pairs $(b_i,c_i)$,  with a smaller hypergraphs, in which each $a_j$ is connected only to $2 q$ such pairs.
For all $j\in J$, define $Q_j := \{ 2q(j-m) - (2q-1),\ldots, 2q(j-m) \}$,
so that $Q_{m+1} = \{1,\ldots,2q\}$, $Q_{m+2} = \{2q+1,\ldots,4q\}$, etc.
Note that each $Q_j$ contains $2q$ indices, and all $Q_j$ are pairwise-disjoint.
Define the following hypergraph:
\begin{equation*}
H_2' = \{a_j b_i c_i \mid j\in J, i \in Q_j\}.
\end{equation*}
Note that $H_2'$ contains $k'\cdot 2q = 2m$ edges.
 
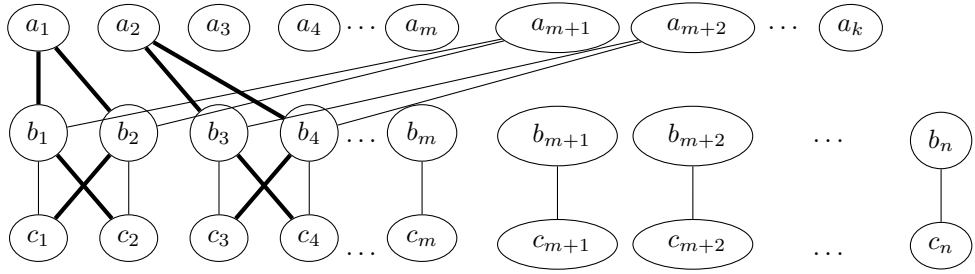
\begin{center}
\def\aline{1}
\def\bline{0.5}
\def\cline{0}
\def\height{0.7}
\begin{tikzpicture}[scale=3]
\node[draw,ellipse] (a1) at (0,\aline) {$a_1$};
\node[draw,ellipse] (a2) at (0.4,\aline) {$a_2$};
\node[draw,ellipse] (a3) at (0.8,\aline) {$a_3$};
\node[draw,ellipse] (a4) at (1.2,\aline) {$a_4$};
\node  at (1.43,\aline) {$\ldots$};
\node[draw,ellipse] (am) at (1.7,\aline) {$a_m$};
\node[draw,ellipse] (am1) at (2.3,\aline) {$~a_{m+1}$};
\node[draw,ellipse] (am2) at (2.9,\aline) {$~a_{m+2}$};
\node  at (3.3,\aline) {$\ldots$};
\node[draw,ellipse] (ak) at (3.6,\aline) {$a_k$};
\node[draw,ellipse,below=\height of a1] (b1)  {$b_1$};
\node[draw,ellipse,below=\height of a2] (b2)  {$b_2$};
\node[draw,ellipse,below=\height of a3] (b3)  {$b_3$};
\node[draw,ellipse,below=\height of a4] (b4)  {$b_4$};
\node  at (1.43,\bline) {$\ldots$};
\node[draw,ellipse,below=\height of am] (bm)  {$b_m$};
\node[draw,ellipse,below=\height of am1] (bm1)  {$~b_{m+1}$};
\node[draw,ellipse,below=\height of am2] (bm2) {$~b_{m+2}$};
\node  at (3.5,\bline) {$\ldots$};
\node[draw,ellipse] (bn) at (4,\bline) {$b_n$};
\node[draw,ellipse,below=\height of b1] (c1) {$c_1$};
\node[draw,ellipse,below=\height of b2] (c2)  {$c_2$};
\node[draw,ellipse,below=\height of b3] (c3)  {$c_3$};
\node[draw,ellipse,below=\height of b4] (c4)  {$c_4$};
\node  at (1.43,\cline) {$\ldots$};
\node[draw,ellipse,below=\height of bm] (cm){$c_m$};
\node[draw,ellipse,below=\height of bm1] (cm1) {$~c_{m+1}$};
\node[draw,ellipse,below=\height of bm2] (cm2) {$~c_{m+2}$};
\node  at (3.5,\cline) {$\ldots$};
\node[draw,ellipse,below=\height of bn] (cn) {$c_n$};
\draw [ultra thick] (a1) -- (b1) -- (c2);
\draw [ultra thick] (a1) -- (b2) -- (c1);
\draw [ultra thick] (a2) -- (b3) -- (c4);
\draw [ultra thick] (a2) -- (b4) -- (c3);

\draw (am1) -- (b1);
\draw (am1) -- (b2);

\draw (am2) -- (b3);
\draw (am2) -- (b4);

\draw (b1) -- (c1);
\draw (b2) -- (c2);
\draw (b3) -- (c3);
\draw (b4) -- (c4);
\draw (bm) -- (cm);
\draw (bm1) -- (cm1);
\draw (bm2) -- (cm2);
\draw (bn) -- (cn);
\end{tikzpicture}
\end{center}

If $n$ is even, define   
$H=H_1 \cup H_2'$. 
To see that $H$ is fractionally balanced, assign weight $\frac{1}{2}$ to every edge of $H_1$ and $\frac{k-m}{2m}=\frac{1}{2q}$ to every edge of $H_2'$.
\HiddenExplanation{
	$\deg a_i = 1$ obviously.
	$\deg b_i = \deg c_i = 1/2 + k'/2m = (k'+m)/2m=k/n$.
}
Let $M$ be any matching in $H$, and assume that some $x\geq 0$ edges of $M$ come from $H_1$.
Each such edge touches a pair of adjacent $(b_i,c_i)$ edges.
There are $m$ such pairs overall, and each edge in $H_2'$ touches a different pair. 
Hence, $M$ can contain at most $m-x$ edges in $H_2'$, so $\nu(H)\le m=\frac{n}{2}$.

If $n$ is odd, let $H=H_1 \cup H_2'\cup H_3$.
Assigning weight $\frac{1}{2}$ to every edge in $H_1$, 
and    $\frac{k}{n(k-m)}$ to every edge of $H_3$, and $\frac{k-m}{n-1}-\frac{k}{n(n-1)}$ to every edge of $H_2'$, we obtain  that $H$ is fractionally balanced.
 Now $\nu(H) \le m+1 = \lceil \frac{n}{2} \rceil$.
\end{proof}


By the previous theorems,
\begin{align*}
\bm(n,2n-2,2n-2)&\leq n-1 && \text{by Theorem \ref{thm:tri-upper-knn}};
\\
\bm(n,2n-1,2n-1)&\geq n && \text{by Theorem \ref{thm:tri-lower}};
\end{align*}
\begin{question}
What is the value of  $\bm(n,2n-2,2n-1)$?
\end{question}
The smallest open case is $\bm(3,4,5)$: we do not know whether it is $2$ or $3$.

The following upper bound uses a different construction.

\begin{theorem*}[Theorem \ref{thm:main_negative}]
Let $r \geq 1$.
Then for any $ 
k \leq rn$,
\begin{align*}
\bm(n,n,k) \leq \left\lceil 2r n/(2r+1) \right\rceil.
\end{align*}
\end{theorem*}
\begin{proof}



Let $n'=\left\lceil 2r n/(2r+1) \right\rceil$. 
We may assume $k > n'$, otherwise
trivially $\bm(n,n,k)\leq k \leq n'$, so
let $N$ be such that  $k=N+n'$.
For such $N$ the system of equations

  \[
    \left\{
                \begin{array}{ll}
                  Ny+(n-n') + n' x=n'\\~\\
                  2(n-n') + n'x =n' y
                  
                \end{array}
              \right.
  \]
have a 
non-negative  
solution in $x$ and $y$: 
$y=\frac{n}{k}$, $x=\frac{n}{k}+2-\frac{2n}{n'}$.

\HiddenExplanation{
$x \geq \frac{1}{r} + 2 - 2\cdot \frac{2r+1}{2r} = 0$.
}

We construct a 3-partite hypergraph $H$.
Let $V
=\{a_1, \dots, a_{N+n'} , b_1, \dots, b_n, c_1, \dots, c_n \}$
Define the hypergraphs on $V$ as follows:
\begin{enumerate}
    \item 
$H_1=\{ a_j b_i c_i
| i \in \left[n' \right], j \in [N] \}$, all weigh $y$;
    \item
    $H_2=\{a_{N+i} b_i c_{n'+j
    }
    |i \in \left[n'  \right], j \in [n-n']\}$, all weigh 1;
    \item
    $H_3=\{a_{N+i} b_{n'+j
    } c_i 
    | i \in \left[n'  \right], j \in [n-n']\}$, all weigh 1;
   \item 
    $H_4=\{a_{N+j}b_ic_i
    | i,j \in  \left[n' \right]\}$, all weigh $x$.
\end{enumerate}
Let $H=H_1 \cup H_2 \cup H_3 \cup H_4$. 

For $i \in [n' ]$ we have $\deg(b_i)=\deg(c_i)=N \cdot  y+(n-n')+n'  \cdot x$, and for 
$j > n'$
we have $\deg(b_{j})=\deg(c_{j})=
n'
$.
\HiddenExplanation{
	Here we need the assumption that $2r$ is an integer.
}
By the choice of $x$ and $y$ these are equal.

For $i \in [N]$ we have  $\deg(a_i)=n'  y$
and 
for $j \in [n']$ we have  $\deg(a_{N+j})=2(n-n')+n' x$. Again, these are equal.
Thus $H$ is fractionally 
balanced.

We claim that $\nu(H) \leq n' $. To see this, note that if $M$ is a matching of size $n'  +1$, then it must contain edges from $H_2 \cup H_3$, and its trace on the $b_i$'s and $c_i$'s should contain two edges of the form $b_i c_{n' + j}$ and 
$b_{n' + j} c_i$, but both these edges can be completed to an edge of $H$ only by adding to them the element $a_{N+i}$. 
\end{proof}

\section{Extending $\bm$ to hypergraphs of higher dimensions}
\label{sec:higher-dimensions}

\subsection{Lower bounds}

What happens to the ``$\bm$'' function if a side of size $n_{d+1}$ is added to the hypergraph? Obviously, adding a coordinate can only (weakly) decrease the value of the function $BM$, but it is natural to assume that adding a large enough coordinate does not strictly decrease it. This is indeed the case.

\begin{theorem}\label{thm:ramseyscr}
If $\bm(n_1, \dots, n_d)\geq m$, then there exists $N = N(n_1,\ldots,n_d)$ such that 
$\bm (n_1, \dots, n_d, n_{d+1}) \geq  m$
whenever $n_{d+1}>N$.
\end{theorem}

The main tool we shall use in the proof of Theorem \ref{thm:ramseyscr} is the so-called ``Gordan's lemma'' from convex geometry. Recall the \emph{dual} of a convex cone $X\subset \R^k$ is 
\[
X^*:=\{v \in \R^k: \left\langle x,v \right\rangle \geq 0 \; \forall x \in X\}.
\]

A polyhedral cone in $\R^k$ is said to be {\em rational} if its extreme rays are multiples of  vectors with rational coordinates.
\begin{theorem}[Gordan's Lemma]\label{thm:Gordan}
If $X\subset \R^k$ is a rational convex polyhedral cone, then  the semigroup $X^* \cap \Z^k$, with the operation of coordinate-wise addition, is finitely generated.
\end{theorem}
Equivalently, the lemma can be stated as follows.
Let $A$ be an integral matrix, and let $X^*=\{\vec{x} \mid A\vec{x} \ge \vec{0}\}$. 
Then there is a finite subset $S\subseteq X^* \cap \Z^k$ such that every vector in $X^* \cap \Z^k$ is a  linear combination of vectors from $S$  with integral coefficients.
Such a subset $S$ is called a \emph{Gordan base} of $X^*$.

Fix positive integers $n_1, \dots, n_d$.
Let $K := \prod_{t=1}^d [n_t] = $ the set of edges in the complete $d$-partite hypergraph with sides of size $n_1, \dots, n_d$,
$ W(n_1, \dots, n_d) \subset \R_{\geq 0}^{K} $
the collection of all balanced weight functions on this hypergraph, 
and
$\bh (n_1, \dots, n_d) $
the collection of $(n_1, \dots, n_d)$- fractionally balanced hypergraphs.
Moreover, for every $H \in \bh(n_1, \dots, n_d)$, let 
\[
W_H:=\{w\in W(n_1, \dots n_d): \supp \; w \subset H\}
\]
be the collection of weight functions on $H$ witnessing its balanced-ness.
The set $W(n_1, \dots, n_d)$ is defined by a set of linear inequalities with integer coefficients:
for each vertex $v \in [n_t]$ in side $t \in [d]$, there are $n_t-1$ inequalities, stating that the sum of weights near $v$ is at least as large as the sum of weights near the other vertices in the same side.
The same is true for its subsets $W_H$.
By definition, this means that these sets are rational convex polyhedral cones. Hence we have:

\begin{claim}
\label{thm:finitebalancedhypergraphs}
For every positive integers $n_1,\ldots,n_d$, the set of integral balanced weight functions,
$W(n_1, \dots, n_d) \cap \Z^K$, is finitely generated as a semigroup.
\end{claim}

\begin{example} For any $n\geq 1$:
\begin{itemize}
    \item 
The $n!$ characteristic functions of perfect matchings form a Gordan base for $W(n,n)$;
this is  the  Birkhoff von-Neumann theorem.

    \item An easy generalization is  that, for any integer $s\geq 1$,  $W(n,sn)$ is generated by the characteristic functions of perfect $(1,s)$-matchings.
    
\end{itemize}
\end{example}

We denote by $\U(n_1,\ldots,n_d)$ the Gordan base of $W(n_1, \dots, n_d)$. If there is more than one, we select one arbitrarily.

\begin{proof}[Proof of Theorem \ref{thm:ramseyscr}]

Choose $n_{d+1}=N$ sufficiently large (to be determined below). 
We assume that every hypergraph in $\bh(n_1, \dots, n_d)$ contains a matching of size $m$; we have to prove that 
the same is true for every hypergraph in $\bh(n_1, \dots, n_d, N)$.

Let $H' \in \bh(n_1, \dots, n_d,N)$.
Define $\cj := [n_1] \times \dots \times [n_d]$ 
and $\cj^+ := \cj\times [N]$.
Then the cone $W_{H'} \subset \R_{\geq 0}^{\cj^+}$ is rational and nonempty, hence it contains a nonzero integral point $w'$.
Now, as $w'$ is balanced, every vertex in the $t$-th class $[n_t]$ has the same total $w'$-degree, say $\delta_t$, for all $t\in[d+1]$. By double-counting, $n_t \delta_t=n_{d+1} \delta_{d+1}=N \delta_{d+1}$ for any $t \in [d]$.

Define $w \in \R_{\geq 0}^{\cj}$ as follows. For each $e \in \cj$, let 
\[ w(e):=\sum_{j\in [N]} w'(e,j). \]
Since $w'$ is balanced, $w$ is balanced too. Specifically, $w \in W_H$, where $H$ is the $d$-partite hypergraph obtained from $H'$ by removing the $(d+1)$-st vertex side $[N]$. 
Since $w\in W_H \subset W(n_1, \dots, n_d)$, 
by Claim \ref{thm:finitebalancedhypergraphs}
there is a finite decomposition
\[
w=\sum_{\ell=1}^T u_{\ell}
\]
for some $u_1, \dots, u_T \in \U(n_1,\ldots,n_d)$
(where if a function $u$ appears in the sum with coefficient $c$, we decompose it as a sum of $c$ copies of $u$). 
Let 
$$
\delta_{\max}:=\max_{
v\in
V(\cj),~
u \in \U(n_1,\ldots,n_d)}\deg_u(v).$$
Since  every vertex $v_t\in [n_t]$ has $w$-degree $\delta_t$, it follows that $\delta_t = \deg_{w}(v_t)
=\sum_{\ell=1}^T \deg_{u_\ell} v_t
 \leq T\cdot \max_{\ell} \deg_{u_{\ell}}(v_t) \leq \delta_{\max}\cdot T$.

Let $H_\ell\in \bh(n_1, \dots, n_d)$ be the support of $u_\ell$ for each $\ell \in [T]$. By assumption,
since $H_\ell$ is fractionally-balanced,
it contains a matching $M_\ell$ of size $m$. The number of different possible such matchings is
\[
q:=
\prod_{t=1}^d \binom{n_t}{m}
\cdot
(m!)^{d-1}.
\]

Now, we let 
$N := \delta_{\max}\cdot q \cdot \max_t n_t $.
Note that $\delta_{\max}$ depends only on the $\{n_i\}$s, hence $N$ depends only on the $\{n_i\}$'s and $m$ (and in particular is independent of the choice of $H$).
Then $T \geq \frac{\delta_t}{\delta_{\max}} =\frac{N \delta_{d+1}}{n_t \delta_{\max}} \geq
q \delta_{d+1}$. The  pigeonhole principle shows that some $\delta_{d+1}$ different  $M_\ell$s are identical.  Without loss of generality $M_1=\dots=M_{\delta_{d+1}}=M$. 
Since $M_{\ell} \subseteq H_{\ell} = \supp(u_{\ell})$, 
we have
 $u_{\ell}(e) \geq 1$ for every $\ell \leq \delta_{d+1}$ and every edge $e \in M$.

For every subset $E\subseteq M$,
let $J_E$ be the subset of vertices $j$ in the $(d+1)$-st part of $H'$ having $w'(e,j) \geq 1$ for some $e\in E$. Such $(e,j)$ are necessarily edges of $H'$.
Then:
\begin{align*}
\delta_{d+1} \cdot |E| 
&\leq \sum_{\ell=1}^{\delta_{d+1}} \sum_{e \in E} u_{\ell} (e) && \text{(since $u_{\ell}(e)\geq 1$ for all $\ell\leq \delta_{d+1}$)}
\\
&\leq \sum_{e \in E} w(e) 
&& \text{(since $w = \sum u_{\ell}$)}
\\
&= \sum_{e \in E}  \sum_{j \in [N]}  w'(e,j)
&& \text{(by definition of $w$)}
\\
&=\sum_{j \in J_E}  \sum_{e \in E}  w'(e,j) 
&& \text{(since $w'$ is nonzero only for $j\in J_E$)}
\\
&\leq |J_E|\cdot \delta_{d+1},
&& \text{(since $\delta_{d+1} = \deg_{w'}(j)$).}
\end{align*}
so $|J_E|\geq |E|$ for all $E\subseteq M$.
Applying Hall's theorem to the bipartite graph on $(M,N)$ induced by $H$,
this implies that there is an injection $g: M \to [N]$ such that $(e,g(e)) \in H$ for every $e \in M$.
This  yields an extension of $M$ to a $(d+1)$-partite matching of size $m$ in $H'$.
\end{proof}

In \citet{alonberman} a geometric proof of Gordan's lemma was given, providing an explicit bound. This can be used to give an upper bound on $N$, but we shall not pursue this. 

\begin{question}
Does there exist $N=O(\prod_{t=1}^d n_t)$ satisfying the conclusion of Theorem \ref{thm:ramseyscr}?
\end{question}

\subsection{Upper bounds}
By Theorem \ref{furedi}, $\bm(n,n,\ldots,n) \ge 2$. We conjecture that this is tight.

\begin{conjecture}\label{conj:nn}
$\bm(\underbrace{n,n\cdots,n}_\text{$n$ times}) =2$.
\end{conjecture}

It suffices to show that for every $n$ there exists a fractionally-balanced $n$-partite hypergraph $H_n$ with sides of size $n$, having $\nu=2$.

One way to construct such $H_n$
is as a union of two intersecting ($\nu=1$) fractionally-balanced $n$-partite hypergraphs with sides of size $n_1$ and $n_2$ with $n_1+n_2=n$. 

Moreover, 
for any integer $t<n$,
an
intersecting  fractionally-balanced
$t$-partite hypergraph can be extended to  
an
intersecting  fractionally-balanced
$n$-partite hypergraph 
by adding $n-t$ duplicates of one of its sides.
Hence, to construct the desired hypergraph $H_n$, it is sufficient to construct two intersecting fractionally-balanced hypergraphs, 
one of which is $t_1$-partite with sides of size $n_1$, and the other is $t_2$-partite with sides of size $n_2$,
such that $t_1,t_2\leq n$ and $n_1+n_2=n$.

One class of intersecting fractionally-balanced hypergraphs is the class of truncated projective planes (projective planes with one vertex deleted). For every integer $q$ for which a projective plane of order $q$ exists, the corresponding truncated projective plane is $q$-partite with sides of size $q-1$.
By combining two such planes, of orders $p$ and $q$, we can get a hypergraph $H_n$ with $n=p+q-2$.
So, we have: 
 
 \begin{observation}
  Conjecture \ref{conj:nn} is true for every $n$ of the form $n=q+p-2$, where $p,q$ are  integers for which projective planes of order $p$ and $q$ exist. In particular, if $p,q$ are prime powers.
 \end{observation}
 
 So, for $n$ even  Conjecture \ref{conj:nn}  would follow from the Goldbach conjecture. It would follow for large $n$ in a similar way from the following: 
 
 \begin{conjecture}
 There exists a function $z(n)\in o(n)$ such that for every $m \le n-z(n)$ there exists an intersecting fractionally balanced $n$-partite hypergraph with sides of size $m$:
  $\bm(\underbrace{m,m\cdots,m}_\text{$n$ times}) =1$.
 \end{conjecture}

\section{An application: multidimensional and rainbow versions of the KKM theorem}
\label{sec:kkm}


Results on the function $\bm$ can be applied to get  versions of the famous KKM theorem (that speaks about a single simplex) for products of simplices. Here are the necessary definitions.

Given a polytope $Q$, 
we denote by $V(Q)$ its set of vertices. The 
$(n-1)$-dimensional simplex $\Delta_{n-1}$ is the set of points $\vec{x}=(x_1, \ldots,x_n)\in \R_{\geq 0}^n$ satisfying $\sum x_i=1$. Its vertices are $e_i, ~i \le n$, where $e_i(j)=\mathbf{1}_{i=j}$, namely $e_i(i)=1, ~ e_i(j)=0$ for $j \neq i$. 

\begin{definition}
Given a polytope $Q$, a KKM-cover for $Q$ is a collection of closed sets $\{A_v \mid v \in V(Q)\}$  satisfying 
 \begin{equation}\label{generalkkm}
     \sigma \subseteq \bigcup_{v \in V(\sigma)}A_v
 \end{equation}
for every face $\sigma$ of $Q$. 
\end{definition}

\begin{remark}
In more general versions, a set $A_\sigma$ is assigned to every face $\sigma$, not only to vertices. See Theorem \ref{thm:komiyaproduct} below.
\end{remark}

A well-known continuous version of  the even better-known Sperner's lemma is:

\begin{theorem}[The KKM theorem \cite{kkm}]\label{kkm}
If $\{A_v | ~v \in V(\dnm)\}$ is a KKM cover  for $\dnm$,
   then $\bigcap_{v\in V(\dnm)} A_v \neq \emptyset$.
\end{theorem}

Gale \cite{gale} proved a rainbow (in another terminology, ``colorful'') version, in which there are $n$ KKM-covers (``colors'') $A^i_v,~~v \in V(\dnm), ~i\in [n]$,  and each contributes a set to the intersecting sub-collection:

\begin{theorem}[Rainbow KKM]\label{gale}
For every $i\in [n]$ let $\ca^i=\big(A^i_v,~~v \in V(\dnm)\big)$  be  a KKM-cover of 
 $\Delta_{n-1}$. Then there exists a bijection $\phi: [n]\to V(\dnm)$ 
such that
 $\bigcap_{i=1}^{n} A^{i}_{\phi(i)} \neq \emptyset$. 
\end{theorem}

We want to extend the theorem from simplices to products of simplices, namely polytopes $Q = \cp := \prod_{t \le d}\Delta_{n_t-1}$, for positive integers $n_t,1\le t \le d$. A vertex of $\cp$ is a tuple $(e_{i_1},\ldots,e_{i_d})$, where $i_t \in [n_t]$ for all $t\le d$. Two vertices 
$(e_{i_1},\ldots,e_{i_d})$ 
and 
$(e_{i'_1},\ldots,e_{i'_d})$ are called {\em disjoint} if $i_t \neq i'_t$ for all $t \le d$.%

\begin{definition}\label{admissible}
For $\cp$ as above, a set $\ca=\{A^1,\dots,A^m\}$, where $A^i = \{A^i_v \mid v\in V(\cp)\}$, of $m$ KKM-covers for $\cp$ is called {\em admissible} if  there exists a function $\phi: [m] \to V(\cp)$ such that 
$\bigcap_{i=1}^m A^i_{\phi(i)}\neq \emptyset$ and the vertices $\phi(i)$ for $i\in[m]$ are pairwise disjoint.
\footnote{
To explain the definition intuitively,
consider each vertex of a simplex as an item,
so each vertex of $\cp$ corresponds to a bundle (a set of items).
Disjoint vertices correspond to bundles that can be allocated to different people.
See Section \ref{sec:cakes} for more details.
}
\end{definition}
Obviously, the maximum possible size of an admissible family 
is $\min_{t \le d}n_t$.
Theorem \ref{gale} says that,
for $d=1$,
every family of $n$ KKM covers for $\Delta_{n-1}$ 
is admissible.
%
%

\begin{notation} 
Let $\ad(n; n_1, n_2, \ldots ,n_d)$ denote the largest integer $m$ such that 
every family of $n$ KKM covers for $\cp := \prod_{t \le d}\Delta_{n_t-1}$ has an admissible sub-family of size $m$.
\end{notation}
In our notation, Theorem \ref{gale} is $\ad(n; n) \geq n$.

\begin{remark}\label{remark:monotone}
The function
$\ad(n; n_1, n_2, \ldots ,n_d)$ is monotone in all arguments. Namely, 
 $\ad(n; n_1, n_2, \ldots ,n_d) \geq \ad(n'; n'_1, n'_2, \ldots ,n'_d)$ whenever  $n'\geq n$, $n'_1 \ge n_1, \ldots ,n'_d \ge n_d$.
 Monotonicity in the first argument is obvious; monotonicity in the other arguments is proved in Proposition \ref{prop:monotonicity-ad}.
\end{remark}

We shall prove the following:

\begin{theorem}\label{thm:bm-implies-ad}$\ad(n; n_1, \ldots ,n_d) \geq \bm(n, n_1, \ldots ,n_d)$
for all integers $d\geq 1$ and $n, n_1,\ldots,n_d\geq 1$.
\end{theorem}

\HiddenExplanation{
\begin{enumerate}
    \item Theorem \ref{thm:bm-implies-ad}. Given a cake-division problem with $d$ cakes, the theorem produces a $(d+1)$-partite hypergraph, with weights on the edges, that are evenly distributed in every vertex side. We shall call such hypergraphs {\em fractionally balanced}. The theorem then asserts that the number of agents that can be satisfied by some  division of the cakes is at least the matching number of this hypergraph. This theorem is inspired by a result of Meunier and Su \cite{MS}.
    
    \item A topological version of Hall's theorem. It will be used to find lower bounds on the matching numbers of fractionally balanced hypergraphs. 
   
\end{enumerate}
}


The   case $d=1$ of the theorem  was proved by Meunier and Su \cite{MS}, using  triangulations. 
This method  can be extended to the general case, but we choose another route, that uses a generalization of the KKM theorem due to Komiya \cite{komiya}. Another beautiful generalization of the KKM theorem that can be used at this point is given in \cite{Zaifu1, Zaifu2}.

\begin{theorem}
\label{thm:komiyaproduct}
Let $R$ be a polytope. 
Let a point $y_\sigma$ be chosen in every face $\sigma$ of $R$  (in particular, a  point $y_R \in R$ is chosen to represent the polytope $R$ itself), and let $B_\sigma$ be a closed subset of $R$ chosen for every face $\sigma$ of $R$. Suppose furthermore that  

	\begin{equation}\label{komiyacondition}
	   \sigma \subseteq \bigcup_{\tau \subseteq \sigma} B_\tau~~ \text{ for every  face}~~\sigma~~\text{of} ~~R.
	\end{equation}
	Then 
there exists a set $Z$ of faces of $R$ 
such that
\begin{enumerate}
\item $\bigcap_{\sigma\in Z} B_{\sigma} \neq \emptyset$. 
\item $y_R \in \conv\{y_{\sigma} | \sigma\in Z\}$.
\end{enumerate}
\end{theorem}

We shall need a slightly more general version.  

\begin{theorem}
\label{thm:komiyaproduct}
Let $R$ be the product $X \times Y$ of two polytopes.
Suppose that in each face $\sigma$ of $R$ there is chosen a point $y_\sigma$, and that for every  nonempty face $\sigma = \alpha \times \beta$ of $R$ ($\alpha$ a face of $X$, $\beta$ a face of $Y$) there is a set $B_\sigma=V_\sigma \times W_\sigma$, 
where 
$V_\sigma$ is an open set in $X$ and $W_\sigma$ is a closed set in $Y$. 
If 
	\begin{equation}\label{komiyacondition}
	   \sigma \subseteq \bigcup_{\tau \subseteq \sigma} B_\tau~~ \text{ for every  face}~~\sigma~~\text{of} ~~R.
	\end{equation}
then 
there exists a set $Z$ of faces of $R$ 
such that
\begin{enumerate}
\item $\bigcap_{\sigma\in Z} B_{\sigma} \neq \emptyset$. 
\item $y_R \in \conv\{y_{\sigma} | \sigma\in Z\}$.
\end{enumerate}
\end{theorem}
 (the original theorem is obtained by putting  $X=\Delta_0$, a single point).
The general version is obtained using a standard technique, of replacing each $ V_\sigma$ by a closed subset of it, while maintaining the intersection pattern of the sets $V_\sigma$ and condition \eqref{komiyacondition}.

\begin{proof}[Proof of Theorem \ref{thm:bm-implies-ad}]
Let $m := \bm(n, n_1,\ldots,n_d)$.
 As before, let $\cp=
 \Delta_{n_1-1}\times \ldots \times \Delta_{n_d-1}$. 
We have a family of $n$ KKM covers of $\cp$,
$\ca^i=\big\{A^i_v,~~v \in V(\cp)\big\}$, for $i\in[n]$.
We have to prove that it has an admissible sub-family of size $m$.

Let  $D$ be a copy of $\Delta_{n-1}$ and $V(D)=\{e_1,\ldots ,e_n\} $. Let $R=D \times \cp$.
We define a Komiya cover $\{B_\sigma \mid \sigma \text{~is a face of~} R\}$ as follows.

For every vertex $v=(e_i,\vj)$ of $R$,  let $B_v=\star(e_k) \times A^i_{\vj}$, where $\star(e_k) := \{v: v_k > 0\}$.
Note that 
$\star(e_k)$ is open
and
$A^i_{\vj}$ is closed.
\HiddenExplanation{
	If we take $\star(e_k) := \{v: v_k \geq 1/n\}$,
	then both sets are closed, and we can use the original theorem.
}

For all other faces $\sigma$ of $R$ (namely faces with a positive dimension) let $B_\sigma=\emptyset$. 
For every  face $\sigma$ of $R$  let $y_\sigma$ be the barycenter of $\sigma$.

\begin{claim}
The sets $B_\sigma$ satisfy
Komiya's condition \eqref{komiyacondition}.
\end{claim}
\begin{proof}
Let $\sigma= \alpha \times \beta$ be a face of $R$ ($\alpha$ a face of $D$, $\beta$ a face of $\cp$) and let $w=(\vec{d}, \vec{p})$ be a point in $\sigma$, where $\vec{d}\in \alpha\subseteq D$ 
and $\vec{p}\in \beta \subseteq \cp$.
We have to show that $w \in B_v$ for some vertex $v \in V(\sigma)$. 
Choose some $i\in[n]$ for which 
$d_i > 0$.
So $\vec{d}\in \star(e_i)$.
Since the sets $A^i_{\vj}$ form a KKM-cover of $\cp$, 
they particularly satisfy the KKM condition \eqref{generalkkm} for its face $\beta$,
so $\vec{p} \in A^i_{\vj}$
for some vertex $\vj\in V(\beta)$.
Then $w \in B_v$ for $v=(e_i, \vj) \in V(\sigma)$. 
\end{proof}

\medskip

By Theorem \ref{thm:komiyaproduct}, there exists a set  $Z \subset V(R)$ 
such that 
\begin{enumerate}
\item $\bigcap_{v\in Z} B_v \neq \emptyset$,
and 
\item $\conv(Z)$ contains the barycenter of $R$. 
\end{enumerate}

Let $H$ be a $(d+1)$-partite hypergraph, in which the $t$-th side, $t \in [d]$, is $V(\Delta_{n_t-1})$, 
and side $d+1$ is $V(D)$.
The edges of $H$ are the elements of $Z$.
Condition (2) above means that some convex combination of the edges of $H$ gives $\frac{1}{n_t}$  on all vertices of side $t,~t\in[d]$, and gives $\frac{1}{n}$ on all vertices of $D$. This means that $H$ is fractionally balanced. 

By assumption, $\bm (n, n_1, \ldots ,n_d) =  m$. So this hypergraph $H$ contains a matching $M=\{h_1,\ldots,h_m\}$ of size $m$, where $h_i = (e_i, \vj_i)$ is an edge of $H$, with 
$e_i \in V(D)$ and $\vj_i \in V(\cp)$.

Now, we define the function $\phi: [m]\to V(\cp)$ as follows:
for all $i\in [m]$, $\phi(i) = \vj_i$. Since $M$ is a matching, the vertices $\phi(i)$ are pairwise-disjoint.
Moreover, since all elements of $M$ are also elements of $Z$, condition (1) above implies that
$\bigcap_{i=1}^m B_{\vj_i} \neq \emptyset$.
So, the sub-family of $m$ KKM covers corresponding to $\vj_1,\ldots,\vj_m$ is admissible.

This concludes the proof that  $\ad(n; n_1, \ldots ,n_d) \geq m$.
\end{proof}

\begin{remark}
\label{rem:ad-not-imply-bm}
The converse of Theorem \ref{thm:bm-implies-ad} is false: $\ad$ may be strictly larger than $\bm$.
This follows from the fact that, 
by our Corollary \ref{cor:bm-upper-bound}
$\bm(2n-1,n,n) \leq \ceil{4n/5}$;
however, \citet{NSZ} have recently proved that 
$\ad(2n-1; n, n) \geq n$ (as we explain in Section \ref{sec:cakes}).
\end{remark}

We do not know if the property of ``higher-dimension extension'' (Theorem \ref{thm:ramseyscr}) holds for  $\ad$. We pose this as a conjecture.

\begin{conjecture}
If $\ad(n; n_1,\ldots,n_d)\geq m$ 
then there exists some $n_{d+1}$ (a function of $n,n_1,\ldots,n_d$)
such that $\ad(n; n_1,\ldots,n_d, n_{d+1})\geq m$.
\end{conjecture}

\section{An equivalent formulation: division of multiple cakes}
\label{sec:cakes}
The rainbow-KKM formulation of ``admissibility'' has an equivalent formulation, using the terminology of ``cakes''. In this section 
we describe this equivalence, for the benefit of those interested in cake-cutting. It can be skipped by those who are content with just the combinatorial formulation. Note, though, that some of the lower bounds on values of the function $\ad$ below are obtained using the terminology of cake partition (which can be then translated into the KKM covers terminology).

In the classic cake-cutting problem \citep{Stromquist80,Su99},
there is a single ``cake'' which is a copy of the unit interval $[0,1]$. A partition of $[0,1]$ into $n$ interval pieces can be identified with the vector $(x_1, \ldots ,x_n)$ of the lengths of the pieces, listed from left to right, and since $\sum x_i=1$, such a partition can be viewed as a point in $\Delta_{n-1}$. 

There is a set of $n$ ``agents'' (or ``players'') and the goal is to divide the cake among them, giving each agent a single interval.
The agents are  choosy. 
 Each agent $i$ has, for each partition $P \in \Delta_{n-1}$,  a nonempty list 
 $L^i(P)\subseteq [n]$ of acceptable pieces, indicated by their indices.
For each index $j\in[n]$, we define
$A^i_{j} := (L^i)^{-1}(j)
= $ the set of partitions (points in $\Delta_{n-1}$) in which agent $i$ accepts piece $j$.
The choices of each agent $i$ should satisfy the following assumptions:

\begin{enumerate}
\item \emph{Closedness}: the sets $A^i_{j}$ are closed for all $j\in[n]$.
\item \emph{Hungriness}: for every partition $P$, the agent accepts some nonempty piece in $P$. That is: 
there exists $j\in[n]$ such that $P_j>0$ and $P\in A_j^i$ (equivalently: $j\in L^i(P)$).
\end{enumerate}

If these two assumptions are satisfied, then there always exists a partition and an assignment of pieces to agents, such that each agent is assigned an acceptable piece. 
We call such a partition an \emph{admissible division}.%
\footnote{In the cake-cutting literature, the lists $L^i$ are called {\em preference lists}.
For each partition $P\in\cp$, $L^i(P)$ is the set intervals that agent $i$ considers the ``best'' in that partition. Then, an admissible division is called {\em envy-free} \citep{Su99}, since an agent who receives a best piece would not envy any other agent.
We prefer the ``admissibility''  terminology, since
the requirement that agent $i$'s portion is in $L^i(P)$ is not a preference, it is absolute. And there is no issue of envy or fairness --- nobody squints at other agents' portions.}
Its existence was proved in several ways \citep{Stromquist80,Su99}. 
In fact, it is equivalent to the rainbow-KKM theorem (\ref{gale}): the ``closedness'' and ``hungriness'' properties are equivalent to the conditions ensuring that $(A^i_j|j\in [n])$ is a KKM-cover for all $i$, and an admissible division is equivalent to a point in the common intersection of $n$ sets from different covers.

The cake-cutting problem can be extended to multiple cakes \citep{cloutier2010two,Lebert2013Envyfree,NSZ}.
Given $d$ ``cakes'' $C_1, \ldots ,C_d$, we consider partitions $P$  
of their union,  the $t$-th cake $C_t$ being partitioned into $n_t$ slices. Then $P=((P^1_1,\dots,P^1_{n_1}),\dots, (P^d_1,\dots,P^d_{n_d}))$ is an element of $\cp:=\prod_{t =1}^{ d}\Delta_{n_t-1}$. 
The subintervals of $C_t$ in the partition $P$ are denoted by  
 $I^t_{j}(P),~ j \in [n_t]$, when ordered from left to right.  So, the length of  $I^t_{j}(P)$
is $P^t_j$.

We denote by $\cj$ the set of all vectors $\vj =(j_1, \ldots ,j_d), ~j_t\in [n_t]$.
For every vector $\vj =(j_1, \ldots ,j_d) \in \cj$, let  
$v(\vj)=(e_{j_1}, e_{j_2}, \ldots e_{j_d})$
be the vertex of $\cp$ corresponding to $\vj$. Our notation will sometimes not distinguish between $\vj$ and $v(\vj)$.  
 
There is a set of agents.
Given a partition $P$, we wish to allocate to each  agent a $d$-tuple  of slices, one from each cake. Such a $d$-tuple is determined by a vertex 
 $v(\vj)$ 
 of 
 $\cp$ (that is, by a vector  $\vj  \in \cj$) ---
 choosing  the slice $I^t_{j_t}$ from $C_t$
 for each $t \in [d]$.  
Of course, we want the $d$-tuples  $\vj_a$ and $\vj_b$ of slices allocated to two distinct agents $a,b$ to be disjoint, namely component-wise distinct. 
As an example application \citep{cloutier2010two}, suppose each ``cake'' represents the time of a workday, and each interval represents a shift. The goal is to assign, to each agent, a shift in every day.

Each agent $i$ has, for each partition $P \in \cp$,  a list 
  $L^i(P)$ of acceptable $d$-tuples $\vec{j}$ of slices, indicated by their indices. 
For example,  suppose $d=3, ~n_1=n_2=n_3=5$.
Then $L^i(P)=\{(3,5,2), (1,4,2) \} $ means that agent $i$  is ready to accept in the partition $P$ either $(I^1_3,I^2_5,I^3_2) $ or $(I^1_1,I^2_4,I^3_2)$. Thus, 
$L^i$ is a multi-valued function from $\cp$ to $\cj$. Its inverse is denoted  by $A^i$. Formally,
for every vector $\vec{j} \in \cj$ and every agent $i$,
\begin{align*}
A^i_{\vec{j}} := (L^i)^{-1}(\vec{j})
=
\{P \in \cp \mid \vec{j} \in L^i(P)\}.
\end{align*}
So $A^i_{\vec{j}}$ is the set of partitions in which agent $i$ is ready to accept the $d$-tuple $\vec{j}$. 

The next observation expresses natural conditions on the lists $L^i(P)$ as a  KKM-cover condition on the sets $A^i_{\vj}$. 
\begin{observation}\label{trivialobs}
For every $i$,  the sets $A^i_{\vj}, ~
\vec{j}  \in V(\cp),$ form a KKM-cover of $\cp$ if and only if they satisfy the following conditions:
\begin{enumerate}
\item \emph{Closedness}: $A^i_{\vj}$ is closed for all $\vec{j}$.
\item \emph{Hungriness}: for every partition $P\in \cp$, 
agent $i$ accepts at least one $d$-tuple of nonempty pieces. In other words,
there exists  $\vec{j} \in V(\cp)$ such that  
    $P^t_{j_t} > 0$ for every $t \in [d]$,
    and $P\in A^i_{\vec j}$ (equivalently: $\vec{j}\in L^i(P)$).
\end{enumerate}
\end{observation}
 

\begin{proof}[Proof of Observation \ref{trivialobs}]
 Suppose that (1) and (2) hold. Let $\sigma$ be a face of $\cp$ and let  $P\in \sigma$. 
By (2), there exists $\vj$ with $P^t_{j_t}>0$ for all $t \in [d]$ and such that $P \in A^i_{\vj}$. Clearly, then, $v(\vj)$ is a vertex of $\cp$ and hence of $\sigma$, showing $P \in \bigcup_{v\in V(\sigma)} A^i_v$. Thus the collection   $\{A^i_v,~v\in V(\cp)\}$ forms a KKM-cover of $\cp$.

Conversely, suppose  that the sets $A^i_{\vj}$ form a KKM-cover of $\cp$. Let $P \in \cp$ and let $S=\supp(P) = $ the minimal face of $\cp$ containing $P$. 
 By (\ref{generalkkm}), there exists a vertex $v=v(\vj)$ of $S$, where $\vj=(j_1,\dots,j_d)$,  with $P \in A^i_{\vj}$. Since $v$ is a vertex in $S$,  $P^t_{j_t}>0$ for all $t\in [d]$, as required in (2).
\end{proof}

Translating Definition \ref{admissible} to the cake-cutting terminology, we say that a set $\ca$ of agents is called {\em admissible} if there
exists a function $\phi: \ca \to V(\cp)$ such that 
$\bigcap_{q \in \ca}A^q_{\phi(q)}\neq \emptyset$ and the vertices
$\phi(q)$ for $q\in\ca$ are pairwise-disjoint.

The condition  means that it is possible
to partition the cakes in a way that placates every player $q \in \ca$. 
A partition $P \in \bigcap_{q \in \ca}A^q_{\phi(q)}$ yields a  division of the cakes, in which if every $q \in \ca$ receives the slices defined by  the vertex $\phi(q)$ then she is happy, since $\phi(q) \in L^q(P)$. We call such an allocation ``admissible''.
We would like to satisfy as many agents as possible, so our aim is to prove the existence of large admissible agent sets.  

Summarizing the discussion above, we have:

\begin{observation}
 $\ad(n;n_1, n_2, \ldots ,n_d) \geq m$ if and only if the following holds:
 
 For every instance of the admissible division problem with $n$ agents and $d$ cakes, where cake $t$ is partitioned into $n_t$ parts, 
there exists a partition $P \in \prod_{t \le d}\Delta_{n_t-1}$ 
for which there exists an admissible set of at least $m$ agents. 
\end{observation}
Note that $n$ and the $n_t$'s play different roles, and are not interchangeable.

\subsection{Previous  results on the values of the function AD} 
Most of the literature on cake-cutting studies the case of a single cake, $d=1$.
\citet{Stromquist80} and \citet{Woodall1980Dividing}, as well as \citet{Su99}, proved that, for any $n\geq 1$, an admissible division exists. In our notation, this means that $\ad(n;n) = n$. 

Recently, some results for $d= 2$ cakes have been proved. These results follow from our Corollary \ref{summarybm} on $\bm$, and from our Theorem \ref{thm:bm-implies-ad} relating $\bm$ to $\ad$.

\begin{itemize}
\item $\ad(2;2,3)  \geq 2$ and $\ad(3;2,2) \geq 2$ {\em \cite{cloutier2010two}}. Follows from  Corollary \ref{summarybm} (\ref{n2-n2}).
\item $\ad(3;5,5) \geq 3$ {\em  \cite{Lebert2013Envyfree}}. Follows from Corollary \ref{summarybm} (4).
\end{itemize} 

A more general result is that, for any $d\geq 2$, $\ad(p; \underbrace{n,\dots,n}_{d ~times}) \geq  \big\lceil\frac{p}{2d(d-1)}\big\rceil$ whenever $p \le d(n-1) +
1$, and 
$\ad(p; \underbrace{n,\dots,n}_{d~times}) \geq  \big\lceil\frac{p}{d(d-1)}\big\rceil$ if $p$ divides  $d(n-1) +
1$ {\em \cite{NSZ}}. 
In particular, 
$\ad(2n-1; n, n) \geq n$. 
This result does \emph{not} have an analogue with $\bm$; see Remark
\ref{rem:ad-not-imply-bm}.

\section{Upper bounds on $\ad$}
\label{sec:upperbounds}
In this section we prove non-existence results for admissible division, implying upper bounds on the function $\ad$.

We prove two upper bounds on $\ad$ for two cakes. 
Both proofs use the same 3-partite hypergraph,  which is based on an example by Drisko \cite{drisko}  (index addition is cyclic, so $n+1\equiv 1$):
\begin{align*}
H_D := &\{ (i, j, j) ~|~ 1\leq i\leq n-1, ~~ 1\leq j\leq n\}
\\
    \cup &\{ (i, j, j+1) ~|~ n\leq i\leq 2n-2, ~~ 1\leq j\leq n\}.
\end{align*}
Note that $H_D$ is $(2n-2,n,n)$-fractionally-balanced and has no matching of size $n$.

In \cite{NSZ} it was proved that  $\ad(2n-1; n, n)\geq n$. The next theorem shows that this is sharp.

\begin{theorem}
\label{thm:2n-2,n,n}
For all $n\geq 2$,
$\ad(2n-2;n,n) <n$.
\end{theorem}


\begin{proof}
We consider an instance of the  two-cake-division  problem with $2n-2$ agents, in which each cake is cut into $n$ slices. We define for each agent $i\in[2n-2]$:
\begin{align*}
C_i &:= \{(j,k) ~|~ (i,j,k) \in H_D\} 
\\
&= 
\begin{cases}
C^1 := \{(j,j) ~|~ j\in [n]\} & \text{~if~} 1\leq i\leq n-1,
\\
C^2 := \{(j,j+1) ~|~ j\in [n]\} & \text{~if~} n\leq i\leq 2n-2 ~~~ [\text{recall that~} n+1\equiv 1]
\end{cases}
\end{align*}
Given partitions $\vec{p},\vec{q}$ of the two  cakes,  we define 
\begin{align*}
B(\vec{p},\vec{q}) &:= \left\{(j,k) ~|~  p_j \geq \frac{1}{n-1},~ q_k \geq \frac{1}{n-1}\right\}.
\end{align*}
The acceptable pairs of agent $i$ are the pairs in $B(\vec{p}, \vec{q})$ and the max-sum pairs in $C_i$:
\begin{align*}
L^i(\vec{p}, \vec{q}) := 
B(\vec{p}, \vec{q})
~\cup~
\{
(j,k)\in C_i: 
p_j + q_k \geq p_{j'} + q_{k'} \text{~for all~} (j',k')\in C_i
\}
.
\end{align*}

First, we show closedness, namely, that for every $i,j,k$, the set
$P_{i,j,k} := \{(\vec{p}, \vec{q}) \mid (j,k) \in  L^i(\vec{p}, \vec{q})\}$ is closed.
If $(j,k)\not\in C_i$, then 
$P_{i,j,k}$ is the set $\{(\vec{p}, \vec{q}) \mid  
 p_j \geq \frac{1}{n-1},~ q_k \geq \frac{1}{n-1}\}$,
 which is closed since it is the intersection of the partition polytope with two closed hyperspaces defined by 
 $p_j \geq \frac{1}{n-1}$ and $q_k \geq \frac{1}{n-1}$.
 If $(j,k)\in C_i$, then $P_{i,j,k}$ is the union of the above set with 
 $\{(\vec{p}, \vec{q}) \mid  
 p_j + q_k \geq p_{j'}+q_{k'}  \text{~for all~} (j',k')\in A_i \}$, which is again defined by intersection of closed hyperspaces.
\HiddenExplanation{
	This argument works if we replace max-sum with max-min, as long as the max-min is taken over $A_i$ only, namaely:
	\begin{align*}
	L^i(\vec{p}, \vec{q}) := 
	B(\vec{p}, \vec{q})
	~\cup~
	\{
	(j,k)\in A_i: 
	\min(p_j , q_k) \geq \min(p_{j'} , q_{k'}) \text{~for all~} (j',k')\in A_i
	\}
	.
	\end{align*}	
}
 
Next, we show hungriness, namely, that for every $i,\vec{p},\vec{q}$, 
the set $L^i(\vec{p}, \vec{q})$ contains at least one pair of nonempty slices.
If $B(\vec{p}, \vec{q})\neq\emptyset$ then it obviously contains (only) pairs of nonempty slices.
If $B(\vec{p}, \vec{q})=\emptyset$, then in at least one cake, say cake 1, all slices are shorter than $1/(n-1)$.
Since their total length is $1$, all slices in that cake are nonempty.
The set $C_i$ contains $n$ pairs, and the total length-sum of all pairs is $2$. Therefore, the maximum length-sum of a pair is at least $2/n$. Since $2/n \geq 1/(n-1)$, and all slices of cake 1 are shorter than $1/(n-1)$, the slice of cake 2 in any pair maximizing the length-sum must be nonempty too.
\HiddenExplanation{
	This argument is much simpler if we replace max-sum with max-min: there is at least one nonempty slice in cake 2, say slice $k$. So the min-length of the pair $(k,k)$ or $(k-1,k)$ is positive. So the max-min is positive. So in the max-min pair, both slices are nonempty.
}

Finally, we prove that in every partition $(\vec{p}, \vec{q})$, at most $n-1$ agents can be allocated an acceptable pair of pieces.

{\bf Case 1:} every agent $i$ gets a pair $(j_i,k_i)\in C_i$.
Since $\{(i,j_i,k_i) \mid i\in [2n-2]\} \subset H_D$, the largest possible matching in this set is of size $n-1$, so at most $n-1$ agents are satisfied. 

{\bf Case 2:} at least one agent gets a pair from $B(\vec{p}, \vec{q}) \setminus C_i$.
The length-sum of this pair is at least $2/(n-1)$.
The length-sum of every max-sum pair in $C_i$ is at least $2/n$. Hence, the length-sum of every $n$ pairs is at least
$\frac{2}{n-1} + (n-1)\cdot \frac{2}{n} 
= \frac{2 n + 2(n-1)^2}{n(n-1)}
= 2 \frac{n^2-n+1}{n^2-n} > 2
$, which is a contradiction since the total length of the cakes is $2$.\end{proof}

\HiddenExplanation{
	This argument does NOT work if we replace max-sum with max-min.
	For example, suppose $n=5$ and the slice lengths are:
	\begin{align*}
	cake~1: & 1/2, 0, 0, 0, 1/2
	\\
	cake~2: & 0,   0, 1, 0, 0  
	\end{align*}	
	Then the max-min is 0, so we can give each agent $i$ any slice from $B\cup A_i$. For example:
	\begin{align*}
		Agent~1: & (1,1) & (A_1)
		\\
		Agent~2: & (2,2) & (A_2)
		\\
		Agent~5: & (3,4) & (A_5)
		\\
		Agent~6: & (4,5) & (A_6)
		\\
		Agent~7: & (5,3) & (B)
	\end{align*}	}

\begin{theorem}
\label{thm:n,n,2n-2}
For all $n\geq 2$,
$\ad(n;n,2n-2) <n$.
\end{theorem}
Note that this is different than Theorem \ref{thm:2n-2,n,n}, since the role of the first argument is different from that of the second and third arguments.

\begin{proof}
We consider a cake-cutting instance with $n$ agents, in which cake \#1 is cut into $n$ slices and cake \#2 is cut into $2n-2$ slices.
We define for each agent $i\in[n]$:
\begin{align*}
C_i  &:= 
\{
(j,k) ~|~ (k,i,j) \in H_D
\}
\hskip 5mm
\text{\small(note the difference from $A_i$ in Theorem \ref{thm:2n-2,n,n})}
\\
&=
\{(i,k) ~|~ 1\leq k \leq n-1 \} 
 \cup
\{(i+1,k) ~|~ n \leq k \leq 2n-2 \} 
\\
&[\text{recall that~} n+1\equiv 1]
\end{align*}

Given partitions $(\vec{p}, \vec{q})$, 
we define 
\begin{align*}
B(\vec{p},\vec{q}) := \{(j,k) ~|~ 
p_{j} \geq \frac{1}{n-1},~
1\leq k \leq 2n-2 \}
\end{align*}
The acceptable pairs of agent $i$ are:
\begin{align*}
L^i(\vec{p}, \vec{q}) := 
&\{
(j,k)\in C_i: 
p_j + q_k \geq p_{j'} + q_{k'} \text{~for all~} (j',k')\in C_i
\}
\\
\cup
&\{
(j,k)\in B(\vec{p}, \vec{q}): 
p_j \geq p_{j'} ,
q_k \geq q_{k'} 
\text{~for all~} (j',k')\in B(\vec{p}, \vec{q})
\}.
\end{align*}

For the proof, we assume without loss of generality  that 
\begin{align*}
q_1 = \max_{1\leq k\leq n-1}q_k
&&
q_n = \max_{n\leq k\leq 2n-2}q_k
\end{align*}
so for all pairs in $L^i(\vec{p}, \vec{q})$, the length of the second slice is either $q_1$ or $q_n$ (for pairs in $B(\vec{p}, \vec{q})$ it is $\max(q_1,q_n)$).

Closedness of the sets $P_{i,j,k}$ can be proved similarly to Theorem \ref{thm:2n-2,n,n}:
If $(j,k)\not\in C_i$, then 
$P_{i,j,k}$ is the set $\{(\vec{p}, \vec{q}) \mid  
 p_j \geq \frac{1}{n-1},~ 
 p_j \geq p_{j'} \text{~for all~} 1\leq j'\leq n,
 q_k \geq q_{k'} \text{~for all~} 1\leq k'\leq 2n-2
 \}$,
 which is closed since it is the intersection of the partition polytope with closed hyperspaces.
 If $(j,k)\in C_i$, then $P_{i,j,k}$ is the union of the above set with 
 $\{(\vec{p}, \vec{q}) \mid  
 p_j + q_k \geq p_{j'}+q_{k'}  \text{~for all~} (j',k')\in C_i \}$, which is again defined by intersection of closed hyperspaces.

Next, we show hungriness.
If $B(\vec{p}, \vec{q})\neq\emptyset$,
then at least one of these pairs, 
with longest slices in both cakes, is in $L^i$, 
and both  slices in this pair are nonempty.
If $B(\vec{p}, \vec{q})= \emptyset$,
then in cake 1 all slices are shorter than $1/(n-1)$.
Since their total length is $1$, all slices in cake 1 are nonempty.
For every max-sum pair $(j,k)\in C_i$, 
$q_k$ equals either $q_1$ or $q_n$. 
If both these lengths are positive then we are done.
If $q_1=0$, then 
$q_2=\ldots=q_{n-1}=0$ too,
so $q_n\geq 1/(n-1)$.
Similarly, if $q_n=0$ then $q_1\geq 1/(n-1)$. In both cases, the max-sum is at least $1/(n-1)$. Since all slices of cake 1 are shorter than $1/(n-1)$, this max-sum can be attained only with a nonempty slice of cake 2.
\HiddenExplanation{
	The hungriness argument is much simpler if we replace max-sum with max-min.
	It works also if we maximize a weighted sum $p_j + r\cdot q_k$, but only if $r\geq 1$.
}

Finally, we show that in every partition $(\vec{p}, \vec{q})$, at most $n-1$ agents can be allocated an acceptable pair. 

{\bf Case 1:} every agent $i$ gets a pair $(j_i,k_i)\in C_i$.
Since  $\{(i,j_i,k_i) \mid i\in [2n-2]\} \subset H_D$, the maximum matching in this set is of size $n-1$, so at most $n-1$ agents are satisfied. 

{\bf Case 2:} at least one agent $i$ gets a pair $(j_i,k_i)\in B(\vec{p}, \vec{q}) \setminus C_i$.
So $p_{j_i}\geq \frac{1}{n-1}$ and $p_{j_i} = \max_{j} p_{j}$ and $q_{k_i} = \max(q_1,q_n)$. Cake 1 must have both large slices (with length at least $\frac{1}{n-1}$) and small slices (with length less than $\frac{1}{n-1}$).
We define a \emph{maximal small-slice sequence} as a  sequence of indices $j_{start},\ldots,j_{end}$ such that $p_j < \frac{1}{n-1}$ for all $j_{start}\leq j\leq j_{end}$ while $p_{j_{start}-1}\geq \frac{1}{n-1}$ and $p_{j_{end}+1}\geq \frac{1}{n-1}$.
Note that the indices in the sequence may be cyclic, e.g. if $n=6$ then the sequence may be $5,6,1,2$.

{\bf Subcase 2.1:} $q_n = q_1$. Then, in every maximal small-slice sequence $j_{start},\ldots,j_{end}$, 
the only agents who are willing to accept a pair with a slice from this sequence are agents $j_{start},\ldots,j_{end}-1$: agent $j_{start}-1$ won't accept slice $j_{start}$ since $p_{j_{start}}+q_n < p_{j_{start}-1}+q_1$, and agent $j_2$ won't accept slice $j_{end}$ since 
$p_{j_{end}}+q_1 < p_{j_{end}+1}+q_n$.
Therefore, at least one of the slices in this sequence remains unallocated.

{\bf Subcase 2.2:} $q_n \neq q_1$; suppose without loss of generality  that $q_1 > q_n$.
At least one agent $i$ must get a pair $(j_i,k_i)$ with $n\leq k_i\leq 2n-2$.
Among all those agents, select one for which 
$j_i$ is smallest.
The pair $(j_i, k_i)$ cannot come from $B(\vec{p}, \vec{q})$ since $q_{k_i}\leq q_n < q_1$ so it is not maximum in cake 2.
Therefore, the pair $(j_i, k_i)$ must come from $C_i$, so we must have $j_i = i+1$.
Also, the sum $p_{j_i}+q_{k_i}$ must be maximum among all pairs in $C_i$. In particular, we must have 
$p_{i+1} + q_n \geq p_{i+1} + q_{k_i} \geq p_i + q_1$.
This implies $p_{i+1} > p_i$.

We now check which agent can be allocated the piece $i$ in cake \#1.
\begin{itemize}
\item Agent $i$ could potentially accept this piece as a part of a pair $(i,k)$ for some $1\leq k\leq n-1$, but $i$ is already allocated another pair.
\item Agent $i-1$ could potentially accept this piece as a part of a pair $(i,k)$ for some $n\leq k\leq 2n-2$. But this is ruled out by the assumption that $j_i$ is smallest.
\item 
Some other agent could potentially accept this piece as a part of a pair in $B(\vec{p},\vec{q})$.
But this is ruled out since $p_i$ is not maximum in cake 1 (since $p_i < p_{i+1}$).
\end{itemize}
In all cases, at most $n-1$ slices of cake \#1 are allocated.
\HiddenExplanation{
	The proof fails if we allow to take all of $B$ (not only the longest pairs). For example, suppose $n=4$ and the slice lengths are:
	\begin{align*}
	cake~1: & 0, 0, 1/3, 2/3
	\\
	cake~2: & 1/3, 1/3, 1/3;  0, 0, 0
	\end{align*}	
	Then we can satisfy all $n$ agents as follows:
	\begin{align*}
		Agent~1: & (1,1) & (A_1: 0+1/3 \geq 0+0)
		\\
		Agent~2: & (2,2) & (A_2: 0+1/3 \geq 1/3+0) 
		\\
		Agent~3: & (4,4) & (A_3: 2/3+0 \geq 1/3+1/3)
		\\
		Agent~4: & (3,3) & (B)
	\end{align*}
}
\end{proof}


\section{Rainbow matchings in families of  $d$-intervals}

 KKM-type theorems have been applied to prove results on  matchings in $d$-interval families, see, e.g., \cite{AKZ, FZ}. But it seems that the fact that there is a simple reduction of $d$-interval matching problems to multiple cake division problems hasn't been  explicitly stated. The purpose of this section is to note this reduction. It  implies that  results on multiple cake-division, in particular those proved above, yield lower bounds on matching numbers in $d$-interval hypergraphs.

Given $d$ disjoint copies $C_1, \ldots, C_d$
of the unit interval $[0,1]$, a $d$-{\em interval} 
is the union of $d$ disjoint open intervals, one on each $C_t$ (the openness is assumed just for simplifying some arguments). 
Let $\F$ be a finite family of $d$-intervals. 
We think of $\F$ as a hypergraph whose vertex set is the uncountable set 
$C_1\cup \cdots \cup C_d$
 and whose edges are the $d$-intervals.
So a matching in $\F$ is a subset of $\F$ consisting of pairwise disjoint $d$-intervals, and a cover in $\F$ is a set of points in the vertex set intersecting all  $d$-intervals in $\F$. 

A well-known theorem of Gallai asserts that when $d=1$, the matching number and the covering number in $\F$ are the same. 
For $d\ge 2$, \citet{tardos} and \citet{kaiser} proved the following: 

\begin{theorem}[\cite{tardos,kaiser}]\label{tardos}
For all $d\ge 2$ and $m\geq 1$, any family of $d$-intervals with matching number $\leq m$ can be covered by  $d(d-1)m$ points, $(d-1)m$ on each $C_t$. 
\end{theorem}
Equivalently, if $\F$ cannot be covered by $d(d-1)m$ points, $(d-1)m$ on each $C_t$, then it has a matching of size $\geq m+1$.
A rainbow version of this theorem was proved in \cite{FZ}: 
\begin{theorem}
\label{coloreddintervals}
Let $d\ge 2$. Let $\F_i,~ i\in[d(n-1)+1]$ be $d(n-1)+1$ families of $d$-intervals 
and write $\F= \bigcup_{i}
\F_i$.
If for all $i \in [d(n-1)+1]$,   $\F_i$ cannot be covered by any choice of $(n-1)d$ points, $(n-1)$ on each $C_t$, then there exists a rainbow matching $\M$  in $\F$  (i.e., $\M$ is a matching in $\F$ and $|\M\cap \F_i | \le 1$) of size $|\M| \ge \ceil{\frac{n}{d-1}}$. 
\end{theorem}
Theorem \ref{tardos} corresponds to the special case in which $n=1+(d-1)m$.
\HiddenExplanation{
Then $(n-1)d = d(d-1)m$,
and $\ceil{n\over d-1} = m+1$.
}

Our results on $\ad$ imply further extensions of Theorem \ref{coloreddintervals}, answering questions of the form: 
Let $\F_1,\dots,\F_n$ be families of $d$-intervals, such that for every $i\in [n]$,  $\F_i$ cannot be covered by $n_t-1$ points on each $C_t$. What is the largest rainbow matching we are guaranteed to find?

\begin{notation}
\label{dintprob}
Let $\im(n; n_1, n_2, \ldots ,n_d)$ denote the largest integer $m$ such that 
any $n$ families  of $d$-intervals $\F_1,\dots,\F_n$, such that  for every $i\in [n]$ $\F_i$ cannot be covered by a choice of $n_t-1$ points on each $C_t$, has a rainbow matching of size $m$.
\end{notation}

In this notation, Theorem \ref{coloreddintervals} is
\begin{align*}
\im\left(d(n-1)+1;\underbrace{n\cdots,n}_\text{$d$ times}\right) \geq \ceil{\frac{n}{d-1}}.
\end{align*}

\begin{theorem}\label{thm:ad-implies-im}
$\im(n; n_1, \ldots ,n_d) \geq \ad(n, n_1, \ldots ,n_d)$
for all integers $d\geq 1$ and $n, n_1,\ldots,n_d\geq 1$.
\end{theorem}
\begin{proof}
Given a collection of families $\F_i$ of $d$-intervals, we construct a $d$-cake-cutting instance 
in which each agent $i$ accepts a $d$-tuple of pieces if and only if it contains a $d$-interval of $\F_i$.
Formally,
for every $(n_1,\ldots,n_d)$-partition $P$ of the cakes, 
let $C(P)$ be the set of its cut-points --- a set containing $n_t-1$ points in each cake $C_t$. By the theorem assumption, 
for every $i\in[n]$, $\cf_i$ is not covered by $C(P)$. This means that 
$\cf_i$ contains at least one $d$-interval $J_i(P)$ that is not cut by the partition. Every such $J_i(P)$ is entirely contained in some $d$-tuple of nonempty intervals from $P$, say
$(I^1_{j_1},\ldots,I^d_{j_d})$.
We let $L^i(P)$ consist of the corresponding vectors $(j_1,\ldots,j_d)$ of indices. 

The closedness of the admissibility sets follows from the fact that the $d$-intervals are open. An admissible division in the cake-cutting instance corresponds to a rainbow matching in the collection of $d$-interval families.
\end{proof}

\begin{remark}

Recall that, by \cite{NSZ},
\begin{align*}
\ad\left(d(n-1)+1;\underbrace{n\cdots,n}_\text{$d$ times}\right)
\geq
 \ceil{\frac{d(n-1) +1}{d(d-1)}} = \ceil{\frac{n}{d-1}-\frac{1}{d}}=\ceil{\frac{n}{d-1}}.
\end{align*}
Combined with Theorem \ref{thm:ad-implies-im} this implies
Theorem \ref{coloreddintervals}.
\end{remark}

As a sample application, here is a 
proof of  the case $d=2$ of Theorem \ref{tardos}, which can be stated as:
in a family of $2$-intervals, if there is no cover containing $2m$ points ($m$ points in each $C_t$), then there is a matching of size $m+1$.

By Theorem \ref{thm:ramseyscr}, since 
$\bm(m+1,m+1)\geq m+1$,  there exists a finite $n$ such that $\bm(n,m+1,m+1) \geq m+1$.%
\footnote{In fact, by Corollary \ref{summarybm} (1),     $n = (m+1)(m+1/2)$ suffices,
but we prefer using Theorem \ref{thm:ramseyscr} because it is a general tool for  proving non-rainbow results from rainbow results.}
By Theorem \ref{thm:bm-implies-ad}, 
$\ad(n;m+1,m+1) \geq m+1$. 
By Theorem 
\ref{thm:ad-implies-im}, $\im(n;m+1,m+1) \geq m+1$.
This holds, in particular, for $n$ families that are all identical to $F$.
So there exists a rainbow matching of size $m+1$ in the union of these $n$  families, which is a matching  of size $m+1$ in $F$.

A similar argument also deduces Theorem \ref{tardos} for general $d$ from F\'{u}redi's theorem --- this is effectively the same argument as in Section 6 of \cite{AKZ}.

Possibly the reduction is only one way. It may well be that the bounds obtained this way are not optimal, namely that $d$-intervals behave better than division of $d$ cakes.
But so far we do not have an example for this. On the contrary, just like $\ad(2;2,2) < 2$, we have
\begin{proposition}
$\im(2; 2, 2)< 2$.
\end{proposition}
\begin{proof}
For simplicity, we scale both cakes to $[0,5]$.
A 2-interval $(x_1,x_2),(y_1,y_2)$ can be visualized as an axes-parallel rectangle $(x_1,x_2)\times (y_1,y_2)$ in the square $[0,5]\times[0,5]$.
We define two families of 2-intervals. Each family contains six 2-intervals, of which two are large (4-by-4) and four are small (1-by-2).

$F_1$ contains 
$\{ (0,4)\times (0,4), (1,5)\times (1,5), (4,5)\times (0,2), (3,5)\times (0,1), (0,2)\times (4,5), (0,1)\times (3,5) \}$;
$F_2$ contains 
$\{ (0,4)\times (1,5), (1,5)\times (0,4), (4,5)\times (3,5), (3,5)\times (4,5), (0,2)\times (0,1), (0,1)\times (0,2) \}$.
In the illustration below, each rectangle is filled with a light color; intersection of two rectangles is filled with darker colors.

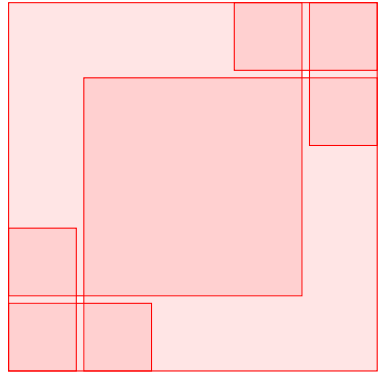
\begin{center}
\newcommand{\fa}[2]{
\fill[blue!50!white,draw=blue,fill opacity=0.2] (#1) rectangle (#2);
}
\newcommand{\fb}[2]{
\fill[red!50!white,draw=red,fill opacity=0.2] (#1) rectangle (#2);
}
\begin{tikzpicture}
\node at (2.5,5.5) {$F_1$:};
\fa{0.05,0.05}{3.95,3.95}
\fa{1.05,1.05}{4.95,4.95}
\fa{4.05,0.05}{4.95,1.95}
\fa{3.05,0.05}{4.95,0.95}
\fa{0.05,4.05}{1.95,4.95}
\fa{0.05,3.05}{0.95,4.95}
\end{tikzpicture}
\hskip 1cm
\begin{tikzpicture}
\node at (2.5,5.5) {$F_2$:};
\fb{0.05,1.05}{3.95,4.95}
\fb{1.05,0.05}{4.95,3.95}
\fb{4.05,3.05}{4.95,4.95}
\fb{3.05,4.05}{4.95,4.95}
\fb{0.05,0.05}{1.95,0.95}
\fb{0.05,0.05}{0.95,1.95}
\end{tikzpicture}
\end{center}

A \emph{cover} of a family by 1+1 points corresponds to two lines, one horizontal and one vertical, that together intersect all six rectangles in the family. 
$F_1$ does not have such a cover:
\begin{itemize}
\item If the vertical line has $x\in (0,1)$ then it does not intersect the rectangles $(1,5)\times(1,5)$ and $(3,5)\times(0,1)$;
\item If the vertical line has $x\in (1,4)$ then it does not intersect the rectangles $(0,1)\times(3,5)$ and $(4,5)\times(0,2)$;
\item If the vertical line has $x\in (4,5)$ then it does not intersect the rectangles $(0,2)\times(4,5)$ and $(0,4)\times(0,4)$.
\end{itemize}
In all three cases, there remain two rectangles that cannot be intersected by a single horizontal line. 
By analogous arguments, $F_2$ does not have a 1+1 cover.

A \emph{rainbow matching} corresponds to a selection of two rectangles, one from $F_1$ and one from $F_2$, such that their projections on both the $x$ and $y$ axes do not intersect. Here no such matching exists: each large 2-interval has a total length of $4+4=8$, while each small 2-interval has a total length of $1+2=3$, so two such 2-intervals have a total length larger than $10$. 
The four small 2-intervals from each family form a blown-up Pasch hypergraph, in which the largest matching is of size $1$.
\end{proof}

\section{Monotonicity}\label{sec:monotonicity}
The following proposition shows that
results on admissible division of cakes are monotone in the sense that 
$\ad(n'_0;n'_1, \ldots ,n'_d) \geq
\ad(n_0;n_1, \ldots ,n_d)$ whenever $n'_i \ge n_i$ for all $i \in [d]$.

\begin{proposition}
\label{prop:monotonicity-ad}
if $n'_i \ge n_i$ for all $i \in [d]$,
then
$\ad(n'_0;n'_1, \ldots ,n'_d) \geq
\ad(n_0;n_1, \ldots ,n_d)$.
\end{proposition}
\begin{proof}
Let $m := \ad(n_0;n_1, \ldots ,n_d)$.
If $n'_0 > n_0$, then 
one can just ignore $n'_0 - n_0$ arbitrary agents and allocate to the remaining $n_0$ agents. 
Using symmetry and induction, it is sufficient to prove that 
$\ad(n_0;n_1, \ldots ,n_d+1)\geq m$.

Intuitively, having the option to make an additional piece cannot hurt, since we can always make the additional piece empty. 
Formally,
let $\{p_i:i \in [n_0]\}$ be a set of agents, each with a choice function defined on all $(n_1, \dots, n_d+1)$-partitions of $d$ cakes. Define a set $\{q_i:i \in [n_0]\} $ of agents with choice functions on all $(n_1, \dots, n_d)$-partitions as follows.

For each $(n_1, \dots, n_d)$-partition 
$Q=((Q^1_1,\dots,Q^1_{n_1}),\dots, (Q^d_1,\dots,Q^d_{n_d}))$,
define an $(n_1, \dots, n_d+1)$-partition 
$P(Q)=((Q^1_1,\dots,Q^1_{n_1}),\dots, (Q^d_1,\dots,Q^d_{n_d},0))$
(so the $(n_d+1)$-st piece in cake $d$ is empty).


For each $i$, the choices of agent $q_i$ in $Q$ are determined by the choices of agent $p_i$ in $P$:
agent $q_i$ accepts in $Q$ the $k$-tuple of slices with indices  $(j_1,\ldots,j_d)$ 
iff $j_d\leq n_d$ and 
agent $p_i$ accepts in $P$ the $d$-tuple of slices with indices  $(j_1,\ldots,j_d)$.
By the hungry agents assumption, and since slice $n_d+1$ is empty in $P$,
agent $p_i$ must accept in $P$ at least one $d$-tuple with $j_d\leq n_d$.
The nonemptiness and continuity conditions for the choice functions of the $q_i$'s follow from those of the $p_i$'s. So by assumption, the $\{q_i\}$ have an admissible division for $m$ of the agents. It directly corresponds to an admissible division for $m$ of the $\{p_i\}'s$.
\end{proof}



We suspect that the  corresponding monotonicity property for $\bm$ is false. The smallest open case is (6,6,5):

\begin{question}
Is it true that $\bm(6,6,5)=4$?
\end{question}

This will refute monotonicity of $\bm$, since $\bm(6,6,6)=3$ by Theorem \ref{thm:nnn}.

\medskip

\section*{Acknowledgments}
We are grateful to Fr\'ed\'eric Meunier and Dmitry Falikman for helpful and instructive discussions.
We also thank users Trebor, Tortar and Piquito from math stackexchange%
\footnote{https://math.stackexchange.com/q/3811677/29780}
for their technical help. 
We are grateful to the anonymous referees of Combinatorica for their detailed comments on the initial version of the paper.


\appendix
\section{Notation summary}

\begin{table}[h]
\begin{tabular}{|c|c||c|c|}
\hline
$n$ & Num. of agents & $i$ & Index of an agent; $i\in[n]$. \\
\hline
$d$ & Num. of cakes / intervals & $t$ & Index of a cake / interval; $t\in[d]$. \\
\hline
$n_t$ & Num. of pieces of cake $t$  & $j_t$ & Piece-index in cake $t$; $j_t\in[n_t]$. \\
\hline
$\cp$
& $=\prod_{t =1}^{ d}\Delta_{n_t-1}$; set of $d$-cake partitions.
& $\cj$ & $=\prod_{t = 1}^d[n_t]$; set of piece-index vectors   \\
\hline
$P^t$
& partition of cake $t$; $P^t\in\Delta_{n_t-1}$ 
& $P$
& partition of $d$ cakes;
$P\in \cp$
\\
\hline
$C_t$ & cake $t$ / interval $t$ ($t\in[d]$) & 
$I^t_{j}(P)$
& 
Interval $j$ in partition $P$ of cake $t$
 \\
\hline
\bm & Matching in frac. balanced hyp. & 
\im & Matching in interval hyp.
\\
\hline
\ad & Admissible cake-division & 
&
 \\
\hline
%
%
\end{tabular}
\end{table}

The following implications hold:

\begin{center}
\begin{tikzpicture}[scale=1]
\node (bm)  at (0,0) {$\bm$};
\node (ad)  at (2,0) {$\ad$};
\node (im)  at (4,0) {$\im$};

\draw [ultra thick, -Triangle] (bm) edge (ad); 
\draw [ultra thick, -Bar]      (ad) edge (bm); 


\draw [ultra thick, -Triangle] (ad) edge (im); 
\end{tikzpicture}
\end{center}
The  $\mapsto$ means that we know that the opposite implication does not hold.
The $\to$ means that it is open whether the opposite implication holds.

\newpage

\bibliographystyle{plainnat}
\bibliography{shira,erel,joe}

\end{document}